\documentclass[11pt]{amsart}
\usepackage{geometry}
\usepackage{amsthm,amsmath,amsfonts,amssymb,graphicx}%,psfrag}
\usepackage{tkz-graph}
\usetikzlibrary{matrix,arrows}

\theoremstyle{definition} \newtheorem{cor}{Corollary}[section]
\theoremstyle{definition} \newtheorem{lem}[cor]{Lemma} 
\theoremstyle{definition}  
\theoremstyle{definition} \newtheorem{thm}[cor]{Theorem}
\theoremstyle{definition} \newtheorem{eg}[cor]{Example}
\theoremstyle{definition} \newtheorem{defn}[cor]{Definition}
\theoremstyle{definition} 
\theoremstyle{definition} \newtheorem*{thma}{Theorem}
\theoremstyle{definition}  
\theoremstyle{definition} \newtheorem*{defa}{Definition}
\theoremstyle{definition}  
\theoremstyle{definition}  
\theoremstyle{definition} \newtheorem*{notationa}{Notation} 
\theoremstyle{definition} %[section]
\theoremstyle{definition} 
\theoremstyle{definition} \newtheorem{rmk}[cor]{Remark}
\theoremstyle{definition} \newtheorem*{rmka}{Remark}
\theoremstyle{definition} 

\newcommand{\N}{\mathbb N}
\newcommand{\up}{\textup}

\newcommand{\X}{\widehat{X}}
\newcommand{\V}{\widehat{V}}
\newcommand{\E}{\widehat{E}}

\DeclareMathOperator{\Mop}{M}
\DeclareMathOperator{\Gop}{G}

\newcommand{\M}{\Mop(\X)}
\newcommand{\G}{\Gop(\X)}
\newcommand{\SPM}[1]{\Mop(#1)}
\newcommand{\SPG}[1]{\Gop(#1)}
\newcommand{\outdeg}{\deg^{+}}

\DeclareMathOperator{\indegree}{indegree}
\DeclareMathOperator{\outdegree}{outdegree}

\newcommand{\MAX}{\textup{MAX}}

\DeclareMathOperator{\supp}{supp}

\DeclareMathOperator{\cl}{cl}
\newcommand{\Z}{\mathbb Z}

\newcommand{\mi}{\iota}

\DeclareMathOperator{\mS}{I}

\DeclareMathOperator{\diam}{diam}

\title[Sandpile Monoids on Directed Graphs]{Algebraic and Combinatorial Aspects of Sandpile Monoids on Directed Graphs}
\author[S.\ Chapman, R.\ Garcia, L.\ D.\ Garc\'{i}a-Puente, M.\ E.\ Malandro, K.\ W.\ Smith]{Scott Chapman,  Rebecca Garcia, Luis David Garc\'{i}a-Puente, \\ Martin E. Malandro, Ken W. Smith}
\address{Postal address for all authors: Box 2206, Department of Mathematics and Statistics, Sam Houston State University, Huntsville, TX 77341-2206, USA}
\email{scott.chapman@shsu.edu, rgarcia@shsu.edu, lgarcia@shsu.edu, malandro@shsu.edu, kenwsmith@shsu.edu}
\date{\today}

\begin{document}

\begin{abstract} The sandpile group of a graph is a well-studied object that combines ideas from algebraic graph theory, group theory, dynamical systems, and statistical physics. A graph's sandpile group is part of a larger algebraic structure on the graph, known as its sandpile monoid. Most of the work on sandpiles so far has focused on the sandpile group rather than the sandpile monoid of a graph, and has also assumed the underlying graph to be undirected. A notable exception is the recent work of Babai and Toumpakari, which builds up the theory of sandpile monoids on directed graphs from scratch and provides many connections between the combinatorics of a graph and the algebraic aspects of its sandpile monoid.

In this paper we primarily consider sandpile monoids on directed graphs, and we extend the existing theory in four main ways. First, we give a combinatorial classification of the maximal subgroups of a sandpile monoid on a directed graph in terms of the sandpile groups of certain easily-identifiable subgraphs. Second, we point out certain sandpile results for undirected graphs that are really results for sandpile monoids on directed graphs that contain exactly two idempotents. Third, we give a new algebraic constraint that sandpile monoids must satisfy and exhibit two infinite families of monoids that cannot be realized as sandpile monoids on any graph. Finally, we give an explicit combinatorial description of the sandpile group identity for every graph in a family of directed graphs which generalizes the family of (undirected) distance-regular graphs. This family includes many other graphs of interest, including iterated wheels, regular trees, and regular tournaments.
\end{abstract}

\maketitle

\section{Introduction}

The sandpile model developed by Bak, Tang, and Wiesenfeld (BTW) is a
mathematical model for dynamical systems that naturally evolve toward
critical states, exemplifying the complex behavior known as
self-organized criticality~\cite{BTW1}. This model has been widely studied and used by various
disciplines, including physics \cite{BTW1,Dhar,Dhar1995}, computer science \cite{BabaiGridTransience,Bjorner2,Bjorner1}, and economics \cite{BiggsChipFiring}. In 1990 Dhar generalized the BTW model into what is now known as the abelian sandpile model \cite{Dhar}. His results awakened the larger
mathematical community to analyze the algebraic, geometric, and
combinatorial structures arising from these dynamical
systems. For example, one of Dhar's early contributions to the theory was his realization that the recurrent states in the sandpile model of a graph form a group, now called the {\em sandpile group} of the graph. The sandpile group of a graph has since been studied by a large number of authors---see, e.g., \cite{BToump, Rotor, LevineTree, LevineLineGraphs,  ToumpakariThesis, ToumpakariTrees}. 

The sandpile group of a graph is a (usually strict) subset of a larger algebraic structure on the graph, called the {\em sandpile monoid} of the graph. Although Dhar introduced the sandpile monoid and the sandpile group of a graph at the same time, much of the subsequent work in the area has focused on the sandpile group. In addition, most work in the area so far has assumed the underlying graph to be undirected. 

The recent paper of Babai and Toumpakari \cite{BToump} forms the first published systematic study of the sandpile monoid of which we are aware. (Some of the work in \cite{BToump} dates back to the 2005 Ph.\ D.\  thesis of Toumpakari \cite{ToumpakariThesis}). In \cite{BToump} Babai and Toumpakari primarily consider directed graphs, and they build up the theory of sandpile monoids on directed graphs from scratch. In this paper we study the sandpile monoid of a finite directed graph, and we extend the existing theory in four main ways. We proceed as follows.

We begin in Section \ref{SecBasicDefs} by reviewing the basic definitions and observations for sandpile monoids and groups on directed graphs. We also recall the results of Babai and Toumpakari \cite{BToump} that connect certain algebraic aspects of the sandpile monoid to combinatorial aspects of the underlying graph that we will need for our development. 

In Section \ref{SubSecMaxSubgroups} we give our first main result, which is a combinatorial description of the maximal subgroups of a graph's sandpile monoid in terms of sandpile groups of easily-identified subgraphs. This is a nontrivial extension of the results of Babai and Toumpakari \cite{BToump} on the idempotent structure of the sandpile monoid, and is accomplished in Theorem \ref{ThmZeroMaxSubgroup} (for the maximal subgroup corresponding to the $0$ of the monoid) and Theorem \ref{ThmNonzeroIdemMaxSubgroup} (for the maximal subgroups corresponding to the nonzero idempotents of the monoid).

In Section \ref{SubSecTwoIdemMonoids} we point out sandpile results for undirected graphs that are really results for sandpile monoids on directed graphs that contain exactly two idempotents (Theorems \ref{ThmRecurrenceEquiv} and \ref{ThmEventuallyRecurrent}).

In Section \ref{SubSecMonoidsNotSPM} we give a new algebraic constraint that sandpile monoids must satisfy (Theorem \ref{ThmUPlusAEqU}) and we exhibit two infinite families of monoids that cannot be realized as sandpile monoids on any graph (Theorems \ref{ThmChainOfpElts} and \ref{ThmTwoIdemsNotSPM}). 

Finally, in Section \ref{SecStronglyRegular} we study a family of directed graphs that generalizes the family of undirected distance-regular graphs (Definition \ref{DefLocallyDistanceRegular}) and we give an explicit combinatorial description of the sandpile group identity for every graph in this family (Theorem \ref{ThmDistRegularId}). Our family includes distance-regular graphs as well as many other graphs of interest, including iterated wheels (Example \ref{ExIteratedWheels}), regular trees (Example \ref{ExRegularTrees}), and regular tournaments (Example \ref{ExRegularTournament}).

\section{Basic algebraic theory for directed graphs}
\label{SecBasicDefs}

In this section we lay out the basic definitions and algebraic theory for sandpiles on directed graphs and we review Babai and Toumpakari's structural results on the sandpile monoid \cite{BToump} that we will need for our development. Other good references include \cite{DharSurvey,Rotor}. We use the notation established in this section throughout the paper.

Let $\N=\{1,2,3,\ldots\}$. Let $\X$ denote a finite weakly connected directed graph (weakly connected means that $\X$ is connected if we replace all directed edges with undirected edges) with vertex set $\V$, directed edge set $\E$, and a distinguished vertex (called the {\em sink}) which is reachable from all other vertices. We allow $\X$ to have loops as well as any number of edges in either direction between pairs of vertices. $X$ denotes the directed graph induced by removing the sink from $\X$. Let $V$ and $E$ be the vertex and directed edge sets of $X$, respectively. We assume that $V$ is nonempty (so that $\X$ has at least one non-sink vertex). For $v\in V$, let $\outdeg(v)$ denote the out-degree of $v$. For $v,w\in \V$, if there exists a path (trivial or not) from $v$ to $w$ in $X$, we write $v\rightarrow w$. In particular, $v\rightarrow v$ for all $v\in \V$. 

In the {\em classical} case, $\X$ is undirected and $X$ is connected. 

A {\em sandpile configuration} $m$ on $\X$ (also called a {\em configuration} or a {\em state}) is a non-negative integer number of ``grains of sand" distributed, possibly unevenly, across the vertices of $V$. Given an ordering of the vertices of $V$, $m$ can be described by a length-$|V|$ vector with entries in $\N \cup \{0\}$, where the $k$th entry of the vector gives the number of grains on the $k$th vertex of $V$.

Let $m$ be a sandpile configuration on $\X$. The {\em support} of $m$, denoted $\supp(m)$, is the set of all $v\in V$ which hold at least one grain of sand.
In the state $m$, a vertex $v\in V$ is said to be {\em stable} if the number of grains on $v$ is less than $\outdeg(v)$, and $m$ itself is said to be stable if every vertex $v\in V$ is stable. If $v$ is unstable, we may {\em topple} $v$, sending one grain of sand along each of its outward edges to its neighbors. This may cause previously stable vertices to become unstable. The sink swallows all grains of sand it receives and never topples, even if it has outward edges. Because there is a path from every vertex to the sink, if we continue toppling unstable vertices we will eventually reach a stable configuration. A sequence of topplings that terminates in a stable configuration is called an {\em avalanche}. Furthermore, the order of topplings in an avalanche is irrelevant---every sandpile configuration has a unique stable configuration \cite{BToump,DharSurvey}.
%\cite[Section 5.2]{DharSurvey}. 

%%%%%%%%%%%%%%%%%%%%%%%%%%%%%%%%%%%%%%%%%
\subsection{The sandpile monoid and sandpile group of a directed graph}

The {\em sandpile monoid} of (or on) $\X$, denoted $\M$, is the collection of all stable sandpile configurations on $\X$. Given two configurations, we may add them pointwise. Denote this operation by $+$. Even if $a$ and $b$ are stable configurations, $a+b$ may not be. For any states $a$ and $b$, denote the stabilization of $a+b$ by $a\oplus b$. The operation on $\M$ is $\oplus$, which is well-defined, associative, commutative, and has an identity \cite{BToump,DharSurvey}. The identity element is the empty configuration, which we denote by $0$. It is evident that
\begin{equation}
|\M| = \prod_{v\in V}\outdeg(v).
\label{EqSizeOfM}
\end{equation}

We follow Babai and Toumpakari \cite{BToump} in our definitions of recurrence and the sandpile group of $\X$. The development in \cite{Rotor} is also equivalent.

\begin{defn}
\label{DefAccessibleInMonoid}
Let $a,b\in \M$. We say that $a$ is {\em accessible} from $b$ (and that $b$ can access $a$) if there exists $x\in \M$ such that $b\oplus x=a$. For $a\in \M,$ we say that $a$ is {\em recurrent} if $a$ is accessible from every element $b\in \M$. If $a$ is not recurrent, $a$ is said to be {\em transient}.
\end{defn}

%Consistent with Toumpakari's thesis (p 24/86)
\begin{defn}Let $(S,\cdot)$ be a semigroup. If $A$ and $B$ are subsets of $S$, we define $A\cdot B=\{a\cdot b:a\in A \textup{ and }b\in B\}$. $I\subseteq S$ is said to be an {\em ideal} of $S$ if $I$ is nonempty, $I\cdot S \subseteq I$, and $S\cdot I \subseteq I$. %Equivalently (easy to show), I is an ideal iff S^1 I S^1 \subseteq I.
Let $L$ be the intersection of all ideals of $S$. If $L$ is nonempty, then $L$ is said to be the {\em minimal ideal} of $S$, and if $L$ is empty then $S$ is said to have no minimal ideal. \end{defn}

%Note that if $I$ and $J$ are ideals of a semigroup $(S,\cdot)$, then $I\cdot J$ is also an ideal and $I\cdot J \subseteq I \cap J$. Therefore, the intersection of a finite number of ideals of $S$ is nonempty. Thus, if $S$ is finite then the minimal ideal of $S$ exists.

If $S$ is a finite semigroup then the minimal ideal of $S$ exists, and if $S$ is also commutative then its minimal ideal is a group \cite{BToump}. 

\begin{defn}The {\em sandpile group} of (or on) $\X$, denoted $\G$, is the minimal ideal of $\M$. 
\end{defn}

By the above comment, $\G$ is a group.
As far as we know, Babai and Toumpakari were the first to give this definition of $\G$. This is a useful definition (or characterization) of $\G$ because it allows many of the basic properties of $\G$ to be derived using straightforward semigroup theory. For several examples, see \cite{BToump}. It is straightforward to show the following equivalent characterization of $\G$ \cite{BToump}.

\begin{thm}The elements of $\G$ are precisely the recurrent elements of $\M$.
\label{ThmRecIffMinimal}
\end{thm}

%\begin{proof}
%If $a$ is recurrent and $g\in \G$, then $a$ is accessible from $g$, so $a=x\oplus g$ for some $x\in \M$. But $\G$ is an ideal, so $a \in \G$. 
%
%On the other hand, suppose $a\in \G$ and $x\in \M$. Since the minimal ideal is the intersection of all ideals, $a$ is in the ideal generated by $x$. That is, $a\in \{x\oplus s|s\in \M\}$, so $a=x\oplus s$ for some $s\in \M$, i.e., $a$ is recurrent.
%
%%proving that the minimal ideal of a commutative semigroup, when it exists, is a group (http://eom.springer.de/M/m063830.htm).
%We now show that $\G$ is a group. The operation on $\G$ is associative because the operation on $\M$ is. Next, let $a,b\in \G$, and consider $e_a$ and $e_b$. These are elements of $\G$ because $a$ and $b$ are. Since $e_a$ and $e_b$ are mutually accessible, we have
%\[
%e_a = e_b \oplus x \up{ and } e_b=e_a\oplus x'
%\]
%for some $x,x' \in \M$. Then
%\[
%e_a = e_b\oplus x = e_b\oplus e_b \oplus x = e_b\oplus e_a = e_a\oplus e_a\oplus x' = e_a\oplus x'=e_b,
%\]
%and hence $\G$ contains exactly 1 idempotent. We denote it by $e$. Furthermore, for $a\in \G$, $a$ is accessible from $e$, so we have $a=e\oplus m$ for some $m\in \M$. Then
%\[
%e\oplus a = e\oplus e\oplus m = e\oplus m = a,
%\]
%so $e$ is the identity for the operation on $\G$. Finally, for $a\in \G$, we have
%\[
%a\oplus \underbrace{a\oplus a\oplus \cdots \oplus a}_k = e
%\]
%for some $k \in \N$, i.e., $\underbrace{a\oplus a\oplus \cdots \oplus a}_k$ is the inverse of $a$ in $\G$.
%\end{proof}

If the sink of $\X$ has any outward edges, denote the graph obtained by deleting them by $\X'$. It is clear that the behavior of any sandpile configuration on $\X$ and $\X'$ is the same, and consequently that $\M\cong \SPM{\X'}$ and $\G\cong \SPG{\X'}$. So, why not just assume at the outset that the sink has no outward edges? It turns out that if $\X$ is classical (meaning that $\X$ is undirected and $X$ is connected), then up to isomorphism $\G$ does {\em not} depend on which vertex is chosen to be the sink. In fact, if $\X$ is Eulerian (meaning that there is a path from $v$ to $w$ in $\X$ for all $v,w\in \V$, and $\indegree(v)=\outdegree(v)$ for all $v\in \V$), then up to isomorphism $\G$ does not depend on which vertex is the sink \cite[Lemma 4.12]{Rotor}. Altering the location of the sink in these cases may, however, alter the structure of $\M$.

%%%%%%%%%%%%%%%%%%%%%%%%%%%%%%%%%%%%%%%%
\subsection{Recurrence and the identity of $\G$}
\label{SubSecRecurrenceAndId}

\begin{notationa}
The totally saturated configuration in $\M$ (that is, the configuration on $\X$ in which every vertex holds as many grains of sand as possible without toppling) will play an important role in what follows. We denote it by $\MAX$. 
\end{notationa}

It is easy to see that $\MAX$ is accessible from any element $m\in \M$---just drop enough grains of sand on each vertex of $V$ to totally saturate them. Hence $\MAX\in\G$. 

The following theorem generalizes directly in its statement from the classical case \cite{DualGraphs} to the directed case. %We are not the first to point these out---see, e.g., \cite{BToump,Rotor,Perkinson}.

%Recall that an {\em idempotent} of a semigroup $(S,\cdot)$ is an element $e$ satisfying $e\cdot e = e$.

\begin{thm}\label{thmThreePartRecurrence}Let $a\in \M$. \begin{enumerate}
	\item $a$ is recurrent if and only if $a$ is accessible from $\MAX$.
\label{thmAccFromMax}
	\item Let $e=[\MAX - (\MAX\oplus\MAX)]\oplus \MAX$. Then $e$ is the identity of $\G$.
\label{ThmMakeTheId}
	\item Let $e$ denote the identity of $\G$. Then $a$ is recurrent if and only if $a\oplus e = a$.
\label{ThmMPlusEEqM}
\end{enumerate}
\end{thm}
\begin{proof}Part \ref{thmAccFromMax}. If $a$ is recurrent, then it is accessible from any configuration. On the other hand, if $a$ is accessible from $\MAX$, write $a=\MAX\oplus m$ for some $m\in \M$. Let $b\in \M$. $\MAX$ is accessible from $b$, so we have $\MAX=b\oplus m'$ for some $m' \in \M$. Then $a=b\oplus m'\oplus m,$
so $a$ is accessible from $b$.

Part \ref{ThmMakeTheId}. $e$ is accessible from $\MAX$ by definition, so $e\in \G$. We show $e$ is idempotent. Let $\delta$ denote $\MAX$.
\begin{align*}
e\oplus e &=
\left([\delta - (\delta\oplus\delta)] \oplus \delta \right) \oplus \left(\delta \oplus [\delta - (\delta\oplus\delta)]\right) \\
&= \left( [\delta - (\delta\oplus\delta)] \oplus (\delta \oplus \delta) \right) \oplus [\delta - (\delta\oplus\delta)]\\
&= \delta \oplus [\delta-(\delta\oplus\delta)]= e,
\end{align*}
as $\left( [\delta - (\delta\oplus\delta)] \oplus (\delta \oplus \delta) \right) = \delta$. Since any group contains exactly one idempotent---its identity---$e$ is the identity of $\G$.

Part \ref{ThmMPlusEEqM}. If $a$ is recurrent, then $a\oplus e=a$ because $a\in \G$ and $e$ is the identity of $\G$. On the other hand, if $a=a\oplus e$, then $a$ is an element of $\G$ because $\G$ is an ideal.
\end{proof}

%NOTE---This corollary and related question are kind of interesting, but it may detract from the focus of this paper.
%\begin{cor}Let $a\in\M$. Then $a\in \G$ if and only if $a$ is accessible from some $g\in \G$.
%\end{cor}
%\begin{proof}
%If $a$ is recurrent, then $a$ is accessible from $\MAX$. On the other hand, if $g\in \G$ and $a$ is accessible from $g$, then we have $g=\MAX \oplus m$ and $a=g\oplus m'$ for some $m,m'\in\M$, so $a=\MAX\oplus m\oplus m'$, i.e., $a$ is accessible from $\MAX$.
%\end{proof}
%
%\begin{question}If $a\in \M$ is not recurrent, then it may still be accessible from some transient states. For $m\in \M$, if we define the {\em accessibility number} $\acc(m)$ to be the number of elements of $\M$ from which $m$ is accessible, then we have $\acc(m)=|\M|$ if and only if $m$ is recurrent. If $m$ is not recurrent, then the largest that $\acc(m)$ can be is $|\M|-|\G|$. For which graphs $\X$ do such $m$ exist? That is, when do we have transient states that are accessible from all other transient states? We carry out some computations in Section HERE---ref. (The 4-wheel is one example.)
%\end{question}

Although part \ref{ThmMakeTheId} of Theorem \ref{thmThreePartRecurrence} gives us an easy way to compute the identity $e$ of $\G$ for any given digraph $\X$, a description of $e$ in terms of the combinatorics of $\X$ can be difficult to obtain. For example, many questions concerning the combinatorial structure of the identity element on the undirected $m\times n$ grid %(where the sink is external to the vertices of the grid, and every vertex on the edge of the grid is connected to the sink by enough edges to make every non-sink vertex have degree 4) 
with an augmented sink remain open \cite{Rotor}. In Section \ref{SecStronglyRegular} we give a combinatorial classification of $e$ for every graph in a family of directed graphs that generalizes the family of undirected distance-regular graphs.

Also, note that parts \ref{ThmMakeTheId} and \ref{ThmMPlusEEqM} of Theorem \ref{thmThreePartRecurrence} combine to provide a straightforward computational test for recurrence of a given state $a\in\M$---namely, if the identity $e$ of $\G$ is not already stored in memory, construct it, and then check whether $a\oplus e$ is $a$. There are tests which are generally computationally faster (in terms of the number of topplings needed), however, which are known as the {\em burning algorithm} for undirected $\X$ \cite{Dhar} and the {\em script algorithm} for directed $\X$ \cite{Speer}.

We also have the following corollary of Theorem \ref{thmThreePartRecurrence}.

\begin{cor}Let $e$ denote the identity of $\G$. Then $a\in\M$ is recurrent if and only if $a$ is accessible from $e$.\label{CorMRecurrentIffAccFromSPId}
\end{cor}
\begin{proof}If $a$ is recurrent, then $a$ is accessible from every element of $\M$. On the other hand, if $a=e\oplus k$ for some $k\in \M$, then $a\oplus e=e\oplus k\oplus e=e\oplus e\oplus k=e\oplus k=a$, and the result follows from part \ref{ThmMPlusEEqM} of Theorem \ref{thmThreePartRecurrence}.
\end{proof}

%%%%%%%%%%%%%%%%%%%%%%%%%%%%%%%%%%%%%%%%
\subsection{Cycles, idempotents, and structure}
\label{SecStructure}

We now describe some of the results of Babai and Toumpakari on the relationship between the algebra of $\M$ and the combinatorics of $\X$ \cite{BToump,ToumpakariThesis}.

%In Section \ref{SubSecCyclesAndAcc} we describe their results on the relationship between cycles, accessibility in $\X$, and accessibility in $\M$. 
%In Section \ref{SubSecIdemStruct} we describe their results on the relationship between the idempotent structure of $\M$ and the cycle structure of $\X$. 
%We will use these results in Section \ref{SubSecMaxSubgroups} to help us describe the structure of the maximal subgroups of $\M$ in terms of sandpile groups on certain subgraphs of $\X$.

%%%%%%%%%%%%%%%%%%%%%

%\subsection{Cycles and accessibility}

\label{SubSecCyclesAndAcc}

In this paper we only consider directed cycles.

\begin{defn}
\label{DefCycle}
A {\em cycle} in a directed graph is a directed path of length 1 or greater from a vertex of the graph to itself.
\end{defn}

\begin{defn}
\label{DefAccessibleInGraph}
 A cycle in $X$ is {\em accessible} from a vertex $v\in \V$ if there is a path from $v$ to a vertex in the cycle using only edges in $E$. In particular, no cycle in $X$ is accessible from the sink.
\end{defn}

We are now using the term {\em accessible} in two contexts: first, in terms of accessibility in the monoid $\M$ (Definition \ref{DefAccessibleInMonoid}), and second, in terms of accessibility in the underlying digraph $\X$ (Definition \ref{DefAccessibleInGraph}). The following theorem shows how these concepts are related.

\begin{thm}\cite{BToump,ToumpakariThesis}. \label{BToumpMainThm1}Let $a,b\in \M$.
\begin{enumerate}
	\item $0$ is accessible from $a$ if and only if $a$ has no grains of sand on any vertex from which a cycle in $X$ is accessible.
\label{LemZeroAccCycle} \label{LemAcyclicAccZero}

	\item If there is a cycle of $X$ on which $a$ contains a grain of sand but on which $b$ does not, then $b$ is not accessible from $a$.
\label{LemCycleGrainAccAB}
	
	\item If $g\in \G$, then $g$ contains at least one grain of sand on every cycle of $X$.
\label{CorGInGOneGrainEveryCycle}
	
	%M=G iff 0 recurrent
\item The following are equivalent.
\label{ThmMGIff0Rec}\label{thmMGAcyclic}
\begin{itemize}
	\item \label{littlelemmapt1} $0$ is recurrent.
	\item \label{littlelemmapt2} $\G=\M$.
	\item \label{littlelemmapt3} $\M$ contains exactly one idempotent.
	\item \label{littlelemmapt4} $X$ contains no cycles.
\end{itemize}

\end{enumerate}
\end{thm}

Note that whenever a vertex on a cycle in $X$ topples during an avalanche, it will leave at least one grain of sand on a vertex somewhere in that cycle, establishing the forward direction of part \ref{LemZeroAccCycle} and also part \ref{LemCycleGrainAccAB} of Theorem \ref{BToumpMainThm1}. For proofs of the remaining parts, see \cite{BToump}.

\label{SubSecIdemStruct}

\begin{notationa} For $a\in \M$, let $e_a$ denote the (unique) idempotent of $\M$ obtained by adding $a$ to itself and stabilizing repeatedly. %We call $e_a$ the ``idempotent class" of $a$.
\end{notationa}

Let us quickly establish that $e_a$ is well-defined.

\begin{lem}
\label{LemIdempotentClass}
Let $(S,\cdot)$ be a finite semigroup, and let $a\in S$. Then there exists $k\in \N$ for which $a^k$ is idempotent. Furthermore, there is only one idempotent in the list $a,a^2,a^3,\ldots$, and it appears infinitely many times in this list.
\end{lem}

\begin{proof}
Consider
$
a,a^2,a^3,\ldots,a^{|S|+1}.
$
These cannot all be distinct, so we have $a^m=a^n$ for some $1\leq m<n\leq |S|+1.$ Thus we can write
\begin{equation}
\label{EqIdempotentConvergence}
a^m=a^m\cdot a^{n-m}.
\end{equation}
Now consider
$
z=(a^m)^{n-m}.
$
Applying \eqref{EqIdempotentConvergence} $m$ times to $z$ expands $z$ into $z\cdot z$, meaning that $z$ is idempotent:
\[
z=(a^m)^{n-m}
=(a^m)^{n-m}\cdot (a^{n-m})^m = z\cdot z.
\]
To finish the proof, suppose $a^k$ and $a^j$ are both idempotent for some $k,j\in \N$. Then
$
a^k=(a^k)^{j}=(a^j)^{k}=a^j,
$
and since $a^k$ is idempotent, we have
$
a^k=a^{2k}=a^{3k}=\cdots.
$
\end{proof}

Since outward edges from the sink in $\X$ do not affect the structure of $\M$, we assume for the rest of the section that there are no outward edges from the sink in $\X$, as it will make some of the following definitions and notation easier to state. %To explain the results of Babai and Toumpakari on the idempotent structure of $\M$, we will need the results from Section \ref{SubSecCyclesAndAcc} and the following notation.

\begin{defn} 
\label{DefClosure}
The {\em closure} of a subset $W\subseteq V$, denoted $\cl(W)$, is given by $$\cl(W)=\{z\in V:\exists w\in W, w\rightarrow z\}.$$
In particular, no closure contains the sink (because $V$ consists of the non-sink vertices of $\X$), and for all $W\subseteq V$ we have $W\subseteq \cl(W)$. %For $v\in V$, we write $\cl(v)$ instead of $\cl(\{v\})$.
\end{defn}

\begin{notationa}
Given a set of non-sink vertices $S\subseteq V$, let $\mi(S)$ denote the subgraph of $\X$ induced by $\cl(S)\cup\{\textup{sink}\}$. That is, $\mi(S)$ is the graph whose vertex set is $\cl(S)\cup\{\textup{sink}\}$ and whose edge set consists of all edges in $\X$ which are directed outward from any vertex in $\mi(S)$ (including edges pointing to the sink). %If $S=\{v\}$ for some $v\in V$, we write $\mi(v)$ instead of $\mi(\{v\})$.
\end{notationa}

Let $S\subseteq V$ be nonempty. Then $\mi(S)$ is a subgraph of $\X$ which is a legal graph for having a sandpile monoid, in the sense that $\mi(S)$ has at least one non-sink vertex and the sink of $\mi(S)$ is accessible from each of its non-sink vertices. Since toppling a state on $\X$ whose support lies in $\mi(S)$ is the same as toppling the same state on $\mi(S)$, we can view the sandpile monoid and sandpile group of $\mi(S)$ as a submonoid and subgroup, respectively, of $\M$ by the natural inclusions. 

\begin{thm}\cite{BToump, ToumpakariThesis}\label{thmEventuallyRecurrent} If $m\in \M$ with $m\neq 0$, then the sequence
\[
m, m\oplus m, m\oplus m\oplus m,\ldots
\]
eventually reaches a recurrent state in the sandpile monoid on $\mi(\supp(m))$. Furthermore, if $e\in \M$ is a nonzero idempotent, then $e$ is recurrent in the sandpile monoid on $\mi(\supp(e))$---that is, $e$ is the identity of the sandpile group on $\mi(\supp(e))$.
\label{CorRecurrentIdemOnItsISupp}
\end{thm}

Note that Lemma \ref{LemIdempotentClass} and Theorem \ref{thmEventuallyRecurrent} together imply that if $m\in \M$ with $m\neq 0$, then the sequence $$m, m\oplus m, m\oplus m\oplus m,\ldots$$ eventually reaches the identity of the sandpile group of $\mi(\supp(m))$.

\begin{defa}
\cite[Section 2.2]{ToumpakariThesis}
Let $W\subseteq V$. $W$ is said to be strongly connected if, for all $w,v \in W$, $w \rightarrow v$. $W$ is a {\em strongly connected component} of $\X$ if $W$ is strongly connected and is not contained in any other strongly connected subset of $V$. A strongly connected component of $\X$ is {\em cyclic} if it contains a cycle and {\em acyclic} otherwise. A vertex $v\in V$ is {\em acyclic} if it does not belong to a cycle.
\end{defa}

Therefore, an acyclic strongly connected component consists of a single, loopless, non-sink vertex, and we consider the sink not to be an element of any strongly connected component. Note that a vertex $v\in V$ is acyclic if and only if $\{v\}$ is a strongly connected component of $\X$ and $v$ is loopless. We will call the cyclic strongly connected components of $\X$ the {\em cyclic strong components of $\X$} for short. Note that a cyclic strong component of $\X$ might contain just one vertex. 
%HERE---include this? The final graph in the proof of Theorem \ref{ThmExactlyKIdems} is an example---in this graph, the cyclic strong components are $\{v_{1}\}, \{v_{2}\}, \ldots, \{v_{k}\}$. 
Note that the adjectives {\em normal} and {\em abnormal} were used instead of {\em cyclic} and {\em acyclic} in \cite{ToumpakariThesis}.

The cyclic strong components of $\X$ form a partially ordered set, where $C_1\leq C_2$ if and only if there is a path in $X$ from some vertex (and hence every vertex) in $C_1$ to some vertex (and hence every vertex) in $C_2$. In this case we say $C_2$ is accessible from $C_1$, and we write $C_1\rightarrow C_2$, so $C_1\leq C_2\iff C_1\rightarrow C_2$.

\begin{defa}Let $P$ be a partially ordered set and let $F\subseteq P$. $F$ is called a {\em filter} of $P$ (or on $P$) if, for all $x,y\in P$, $x\in F$ and $y\geq x$ implies $y\in F$. 
\end{defa}

\begin{rmka}Our use of the word {\em filter} differs slightly from the normal order theoretic usage. In particular, we do not require our filters to have the property that for all $x,y \in F$, there exists $z\in F$ such that $z\leq x$ and $z\leq y$.
\end{rmka}

%\begin{lem}Let $e\in \M$ be idempotent with $e\neq 0$. Then $\mi(\supp(e))$ contains a cycle in $X$. 
%\label{LemSuppHasCycle}
%\end{lem}
%\begin{proof}By Theorem \ref{thmMGAcyclic}, if $\mi(\supp(e))$ doesn't contain a cycle in $X$, then the sandpile monoid and group on $\mi(\supp(e))$ would coincide, in which case by Theorem \ref{ThmMGIff0Rec} we would have $e=0$. 
%\end{proof}

%So, if $e\neq 0$ is an idempotent in $\M$, then $\mi(\supp(e))$ contains at least one normal strong component of $\X$.  

The following theorem, which is part of the main result of \cite{BToump}, tells us how to read the idempotent structure of $\M$ from the cycle structure of $\X$.

%HERE---specific thereom number in BToump?
%\begin{thm}\cite[Proposition 4.3.17]{ToumpakariThesis}, \cite[Theorem 4.13]{BToump}
\begin{thm}
\label{ThmBijection}
The idempotents of $\M$ are in bijective correspondence with the filters on the set of cyclic strong components of $\X$. Specifically, let $E$ denote the set of idempotents of $\M$ and let $Q$ denote the set of filters on the cyclic strong components of $\X$. Let $e\in E$. Define $$F:E\rightarrow Q$$ by $F(0)=\varnothing$ and, if $e\neq 0$, then $F(e)$ is the set of cyclic strong components of $\X$ contained in $\mi(\supp(e))$. In other words, if $C$ is a cyclic strong component of $\X$, then $C\in F(e) \iff C\subseteq \mi(\supp(e))$. Then $F$ is a bijection.
\end{thm}

\begin{rmk}
\label{RmkNoInitialAbnormalForIdem}
Thus, given a nonempty filter $q=\{C_1,C_2\ldots,C_n\}$ on the set of cyclic strong components of $\X$, there is exactly one nonzero idempotent $e$ of $\M$ whose support intersects each of the $C_1,C_2,\ldots,C_n$ and none of the other cyclic strong components of $\X$. It is straightforward to show that if an acyclic vertex $v$ is not accessible from at least one of the $C_i\in q$, then $v\notin\supp(e)$ 
%\cite[Lemma 6.8]{BToump}. 
\cite{BToump}. 
(However, we note that in general, acyclic vertices which are accessible from the $C_i$ may be contained in $\supp(e)$---see, e.g., the idempotents $e_2$, $e_3$, $e_4$, and $e_5$ in Example \ref{TheExample}.) Thus by Theorem \ref{thmEventuallyRecurrent}, we see that $e$ can be found by placing a grain of sand on any vertex of each $C_i$ and adding the resulting state to itself and stabilizing over and over until an idempotent is reached. 
%That is, Notice in the proof we have also shown that if $e\in\M$ is idempotent with $e\neq 0$, then $\mi(\supp(e))=\cl(\cup_{C\in F(e)}C) \cup \{\textup{sink}\}$, meaning that the vertices of $\supp(e)$ ``farthest away" from the sink lie on cycles of $\X$. 
However, by Theorem \ref{CorRecurrentIdemOnItsISupp}, $e$ is the the identity of the sandpile group of the graph $Y=\mi(\supp(e))=\cl(\cup_{C\in q}C) \cup \{\textup{sink}\}$, so a better way to find $e$ is to use part \ref{ThmMakeTheId} of Theorem \ref{thmThreePartRecurrence}. Specifically, if we let $\MAX_{Y}$ denote the totally saturated configuration on $Y$, then we have
\[
e=[\MAX_{Y}-(\MAX_Y\oplus\MAX_Y)]\oplus\MAX_Y.
\]
\end{rmk}

\begin{eg}
\label{TheExample}
Let $\X$ be the graph in Figure~\ref{IdempotentExample}. 

\begin{figure}[h] 
\begin{center}
\begin{tikzpicture}[>=stealth']
\tikzset{VertexStyle} = [shape = circle]
\SetUpVertex[Lpos=-90]
\SetUpEdge[lw= 1.2pt]
\SetGraphUnit{1.5} 
%\GraphInit[vstyle=Classic] 
\Vertex[Math,L=a_{1},Lpos=-135]{a1}
\NOEA[Math,L=b_{3},Lpos=90](a1){b3}
\SOEA[Math,L=b_{1}](a1){b1}
\NOEA[Math,L=b_{2},Lpos=-45](b1){b2}
\EA[L=sink,Lpos=-90](b2){s}
\NO[Math,L=a_{2},Lpos=90](s){a2}
\NOEA[Math,L=d_{2},Lpos=90](s){d2}
\SOEA[Math,L=d_{1}](s){d1}
\EA[Math,L=d_{3},Lpos=0](d2){d3}
\EA[Math,L=d_{4}](d1){d4}
\NOEA[Math,L=c_{1},Lpos=180](d4){c1}
\NOEA[Math,L=c_{4},Lpos=90](c1){c4}
\SOEA[Math,L=c_{2}](c1){c2}
\NOEA[Math,L=c_{3},Lpos=0](c2){c3}
\NOEA[Math,L=a_{3}](c3){a3}
\tikzset{EdgeStyle/.style={->}}
\Edge[style={bend left}](a1)(b3)
\Edge[style={bend right}](a1)(b1)
\Edge(b3)(b1)
\Edge[style={bend left}](b1)(b3)
\Edge(b3)(b2)
\Edge[style={bend right}](b2)(b3)
\Edge(b1)(b2)
\Edge[style={bend left}](b2)(b1)
\Edge(b2)(s)
\Edge[style={bend right}](a2)(s)
\Edge[style={bend left}](a2)(s)
%%%%%%%%\Edge[style={bend left}](d2)(s)
\Edge(d2)(s)
\Edge(d1)(s)
\Edge(d1)(d2)
\Edge(d4)(d1)
\Edge(d3)(d2)
\Edge(d3)(d1)
\Edge(d4)(d2)
\Loop[dist=1cm,dir=NO,style={very thick,->}](d3)
\Edge(c1)(d3)
\Edge(c1)(d4)
\Edge(c4)(c1)
\Edge[style={bend left}](c1)(c4)
\Edge(c1)(c2)
\Edge[style={bend left}](c2)(c1)
\Edge(c2)(c3)
\Edge[style={bend left}](c3)(c2)
\Edge(c3)(c4)
\Edge[style={bend left}](c4)(c3)
\Edge(c3)(c1)
\Edge(c2)(c4)
\Edge(a3)(c4)
\Edge(a3)(c3)
\end{tikzpicture}
\end{center}
\caption{$\X$ for Example \ref{TheExample}.}\label{IdempotentExample}
\end{figure}

The acyclic vertices are $a_1, a_2, a_3, d_1, d_2,$ and $d_4$. 
The sets $B= \{ b_{1}, b_{2}, b_{3} \}$, $C=\{ c_{1}, c_{2}, c_{3}, c_{4} \}$, and $D = \{ d_{3} \}$ are the cyclic strong components. $P=\{ B, C, D \}$ is the poset of cyclic strongly connected components, where $C < D$ and $B$ is incomparable to $C$ and $D$. The Hasse diagram of $P$ is given in Figure~\ref{Hasse}. Note that, since $\outdeg(d_2)=1,$ $d_2$ holds no grains of sand in any stable configuration on $\X$.

\begin{figure}[h] 
\begin{center}
\begin{tikzpicture}
\SetUpVertex[Lpos=-90]
\SetUpEdge[lw= 1.2pt]
\SetGraphUnit{1.5} 
\tikzset{VertexStyle} = [shape = circle]
%\GraphInit[vstyle=Classic] 
\Vertex[Math,L=C]{c}
\NOEA[Math,L=D](c){d}
\SOEA[Math,L=B](d){b}
\tikzset{EdgeStyle/.style={-}}
\Edge(c)(d)
\end{tikzpicture}
\end{center}
\caption{Hasse diagram of the strongly connected components of $\X$.}\label{Hasse}
\end{figure}
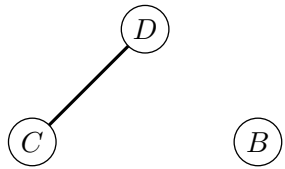

The filters on $P$ are $F_{0} = \{  \}$, $F_{1} = \{B  \}$, $F_{2} = \{ D  \}$, $F_{3} = \{B, D  \}$,  $F_{4} = \{ C,D \}$, $F_{5} = \{B,C,D  \}$. Each filter corresponds to an idempotent in the following way: 

\begin{itemize}
	\item $F_{0}$ corresponds to $0$. %correct
	\item $F_{1}$ corresponds to the idempotent $e_{1}$ where vertices $b_{1}$ and $b_3$ each have 1 grain, $b_{2}$ has 2 grains, and all other vertices have zero grains. %correct
	\item $F_{2}$ corresponds to the idempotent $e_{2}$ where $d_1$ has 1 grain, $d_3$ has 2 grains, and all other vertices have zero grains. %correct
	\item $F_{3}$ corresponds to the idempotent $e_{3}=e_1\oplus e_2.$ %correct 
	\item $F_{4}$ corresponds to the idempotent $e_{4}$ where $c_1,c_3,$ and $d_1$ each have 1 grain, $c_2$ and $d_3$ each have 2 grains, and all other vertices have zero grains. %correct
	\item $F_{5}$ corresponds to the idempotent $e_{5} = e_1\oplus e_4$, which is the identity of $\G$. 
\end{itemize}

\end{eg}

\section{The maximal subgroups of $\M$}

\label{SubSecMaxSubgroups}

In this section we state and prove the first of our main results, which is a combinatorial classification of all the maximal subgroups of $\M$. We prove in Theorems \ref{ThmZeroMaxSubgroup} and \ref{ThmNonzeroIdemMaxSubgroup} that %, with the possible exception of the maximal subgroup containing $0$, 
they are all sandpile groups on certain subgraphs of $\X$ which are easy to identify. In this section we assume the sink of $\X$ has no outward edges. 

\begin{defa}A {\em subgroup} of a semigroup is a nonempty subset of the semigroup which is a group under the semigroup operation. A subgroup of a semigroup is {\em maximal} if it is not properly contained in any other subgroup of the semigroup.
\end{defa}

The following is a well-known result in semigroup theory \cite[Section 1.7]{CliffPres}. 

\begin{thm}Let $(S,\cdot)$ be a semigroup. Each idempotent $e\in S$ is the identity for a unique maximal subgroup $G_e$ of $S$, and $G_e$ is the group of units of the monoid $\{e\}\cdot S\cdot\{e\}$, which has $e$ as its identity. Every subgroup of $S$ which contains $e$ is contained in $G_e$.
\end{thm}

In the language of sandpile monoids, we have:

\begin{cor} 
\label{ThmCliffPresMaxSubgp}
Each idempotent $e\in \M$ is the identity for a unique maximal subgroup $G_e$ of $\M$, and $G_e$ is the group of units of the monoid $\{e\oplus m:m\in \M\}$, which has $e$ as its identity. Every subgroup of $\M$ which contains $e$ is contained in $G_e$.
\end{cor}

Given an idempotent $e\in \M$, let us denote by $G_e$ the corresponding maximal subgroup of $\M$. $\G$ is a subgroup of $\M$. In fact it is a maximal subgroup:

\begin{thm}$\G$ is a maximal subgroup of $\M$.
\end{thm}

\begin{proof}Let $e$ be the identity of $\G$. $\G$ is a subgroup of $\M$ containing $e$, so $\G \subseteq G_e$. To explain why $G_e \subseteq \G$, let $g\in G_e$. Then, since $e\in \G$ and $\G$ is an ideal, $g\oplus e=g$ is in $\G$.
\end{proof}

We already know from the results of Section \ref{SubSecIdemStruct} that, given a nonzero idempotent $e\in \M$, $e$ is the identity for the sandpile group on $\mi(\supp(e))$, and this sandpile group is a subgroup of $\M$ under the natural inclusion. However, this group is not necessarily the {\em maximal} subgroup in $\M$ containing $e$. To give a combinatorial description of $G_e$ we will use the following lemma.

\begin{lem} Let $e$ be an idempotent of $\M$ and let $a\in\M$. Then $a\in G_e$ if and only if $a$ and $e$ are mutually accessible; that is, $a\in G_e$ if and only if $a=e\oplus m$ and $e=a\oplus k$ for some $k,m\in \M$.
\end{lem}

\begin{proof}We have from Corollary \ref{ThmCliffPresMaxSubgp} that $G_e$ is the group of units of $\{e\oplus m:m\in \M\}$, i.e.,
\[
G_e=\{a\in\M:a=e\oplus m \textup{ and }a\oplus(e\oplus k)=e \textup{ for some }m,k\in\M\}.
\]
By definition, if $a\in G_e$ then $a$ and $e$ are mutually accessible. To show the reverse, suppose $a$ and $e$ are mutually accessible. Then $a=e\oplus m$ for some $m\in \M$ and $e=a\oplus k$ for some $k\in \M$. Adding $e$ to each side of the last equation and using that $e$ is idempotent, we obtain $a\oplus (e\oplus k)=e$.
\end{proof}

We now describe the subgraph of $\X$ whose sandpile group will turn out to be $G_e$.

\begin{notationa}Let $e\in\M$ be idempotent. Let $A(e)$ denote the set of acyclic vertices $v\in V$ for which $v\notin \cl(\supp(e))$ and for which the only cyclic strong components of $\X$ accessible from $v$ are those which intersect $\supp(e)$ (i.e., those which are contained in $\mi(\supp(e))$).
%Recall from Theorem \ref{ThmBijection} that $F(e)$ denotes the set of normal strong components of $\X$ accessible from $\supp(e)$. 
Let $\mS(e)$ denote the subgraph of $\X$ induced by
\[
\cl(\supp(e))
\cup
A(e)
\cup
\{\textup{sink}\}.
\] %and whose edge set consists of all edges in $\X$ which are directed outward from any vertex in $\mS(e)$ (including edges pointing to the sink). 
\end{notationa}

Notice that, provided $\mS(e)\neq\{\textup{sink}\}$ (in particular, when $e\neq 0$), $\mS(e)$ is a subgraph of $\X$ which is a legal graph for having a sandpile monoid, in the sense that $\mS(e)$ has at least one non-sink vertex and the sink of $\mS(e)$ is accessible from each of its non-sink vertices. On the other hand, if $\mS(e)=\{\textup{sink}\}$, which is only possible if $e=0$, then the only thing preventing $\mS(e)$ from being a legal graph for having a sandpile monoid is that it has no non-sink vertices. We therefore deal with $G_0$ first.

\begin{thm}
\label{ThmZeroMaxSubgroup}
If $\mS(0)=\{\textup{sink}\}$, then $G_0=\{0\}$. Otherwise, $G_0$ is equal to the sandpile group on $\mS(0)$, viewed as a subgroup of $\M$ by the natural inclusion.
\end{thm}
\begin{proof}If $\mS(0)=\{\textup{sink}\}$, then there is a path to a cycle of $X$ from every vertex $v\in V$. By part \ref{LemZeroAccCycle} of Theorem \ref{BToumpMainThm1}, then, no nonzero state of $\M$ can access 0, so $G_0=\{0\}$. 

On the other hand, if $\mS(0)\neq\{\textup{sink}\}$, then $\mS(0)=A(0)\cup\{\textup{sink}\}$. Let $a\in \M$ be any state whose support is contained in $\mS(0)$. Certainly $a$ is accessible from $0$, and $0$ is accessible from $a$ by part \ref{LemAcyclicAccZero} of Theorem \ref{BToumpMainThm1}. If $a\in \M$ is any state with a grain outside of $\mS(0)$, then $a$ cannot access $0$ by part \ref{LemZeroAccCycle} of Theorem \ref{BToumpMainThm1}. Thus $G_0=\{a\in \M: \supp(a)\subseteq \mS(0)\}$, the sandpile monoid of $\mS(0)$. This is equal to the sandpile group on $\mS(0)$ by part \ref{ThmMGIff0Rec} of Theorem \ref{BToumpMainThm1}.
\end{proof}

We now deal with the nonzero idempotents. The following theorem allows us to build the $G_e$ from the sandpile groups of the $\mi(\supp(e))$.

\begin{thm} 
\label{ThmBuildTheMaxSubgp}
Let $e\in\M$ be idempotent with $e\neq 0$. Then $$G_e=\{g\oplus j: g,j\in \M\textup{ where }g\textup{ is recurrent on } \mi(\supp(e)) \textup{ and }\supp(j)\subseteq A(e)\}.$$
\end{thm}

\begin{proof}We handle the reverse containment first. Consider $a=g\oplus j$, where $g$ is recurrent on $\mi(\supp(e))$ and $\supp(j)\subseteq A(e)$. Then $a=g\oplus j=(e\oplus g) \oplus j$, so $a$ is accessible from $e$. By part \ref{CorGInGOneGrainEveryCycle} of Theorem \ref{BToumpMainThm1}, $g$ has a grain of sand on every cycle of $\mi(\supp(e))$. The support of $g$ and the support of $j$ do not intersect, so $g\oplus j=g+j$. Thus the cyclic strong components of $\X$ accessible from the vertices of $\supp(a)$ are precisely those contained in $\mi(\supp(e))$. By Theorem \ref{thmEventuallyRecurrent} and Remark \ref{RmkNoInitialAbnormalForIdem}, then, $e_a=e$, i.e., $e$ is accessible from $a$.

Now, suppose $a\in G_e$, i.e., $a$ and $e$ are mutually accessible. Thus $a=e\oplus m$ and $e=a\oplus k$ for some $m,k\in \M$. Write $a=g\oplus j$ for the unique states $g,j\in \M$ such that $\supp(g)\subseteq \mi(\supp(e))$ and $\supp(j)\cap \mi(\supp(e)) = \varnothing$. We claim that $g$ is recurrent on $\mi(\supp(e))$ and $\supp(j)\subseteq A(e)$. 

To see why, first, $a=e\oplus m=g\oplus j$, so $a\oplus e = e\oplus e\oplus m = g\oplus j\oplus e$, but $e\oplus e\oplus m = e\oplus m=a$, so $g\oplus j\oplus e = g\oplus j$. Since $\supp(j)\cap\mi(\supp(e))=\varnothing$, $g\oplus e=g$, i.e., $g$ is recurrent on $\mi(\supp(e))$.
Finally, if $\supp(j) \nsubseteq A(e)$, then $j$ has a grain of sand on a vertex from which some cyclic strong component $C\nsubseteq \mi(\supp(e))$ of $\X$ 
is accessible via a path $p$ which does not intersect $\mi(\supp(e))$. Note that $C$ does not intersect $\supp(e)$. We have $a=g\oplus j$ and $e$ is accessible from $a$, so $e=g\oplus j\oplus k$ for some $k\in\M$. But then $\supp(e)\neq \supp(g\oplus j\oplus k)$ because $g\oplus j\oplus k$ must have a grain somewhere along $p$ or on $C$. Hence $\supp(j)\subseteq A(e)$.
\end{proof}

\begin{thm}
\label{ThmNonzeroIdemMaxSubgroup}
Let $e\in \M$ be idempotent with $e\neq 0$. Then $G_e$ is equal to the sandpile group on $\mS(e)$, viewed as a subgroup of $\M$ by the natural inclusion.
\end{thm}

\begin{proof}We must show that $G_e$, as described in Theorem \ref{ThmBuildTheMaxSubgp}, is the set of recurrent elements of the sandpile monoid on $\mS(e)$. Let $a\in G_e$. Then $a=g\oplus j$ where $g$ is recurrent on $\mi(\supp(e))$ and $\supp(j)\subseteq A(e)$. By definition, $a$ is in the sandpile monoid of $\mS(e)$. Since $e$ is the identity of the sandpile group on $\mi(\supp(e))$, $a\oplus e=(g\oplus e)\oplus j=g\oplus j = a$, and since $e$ is the identity of the sandpile group on $\mS(e)$, by part \ref{ThmMPlusEEqM} of Theorem \ref{thmThreePartRecurrence}, $a$ is recurrent in the sandpile monoid on $\mS(e)$. For the reverse inclusion, the sandpile group of $\mS(e)$ is a subgroup of $\M$ which contains $e$, so it is contained in $G_e$ by the maximality of $G_e$.
\end{proof}

\begin{eg} (Example \ref{TheExample} continued) Consider the graph $\X$ in Figure \ref{IdempotentExample}. Corresponding to the idempotents $e_0$ through $e_5$ of Example \ref{TheExample} we have the following maximal subgroups:
\begin{itemize}
	\item $G_{e_0}=G_0$ is the sandpile group on $\{a_2,d_1,d_2,d_4,\textup{sink}\}$.
	\item $G_{e_1}$ is the sandpile group on $B\cup\{a_1,a_2,d_1,d_2,d_4,\textup{sink}\}$.
	\item $G_{e_2}$ is the sandpile group on $D\cup\{a_2,d_1,d_2,d_4,\textup{sink}\}$. Note $A(e_2)=\{a_2,d_4\}$.
	\item $G_{e_3}$ is the sandpile group on $B\cup D\cup\{a_1,a_2,d_1,d_2,d_4,\textup{sink}\}$.
	\item $G_{e_4}$ is the sandpile group on $C\cup D\cup\{a_2,a_3,d_1,d_2,d_4,\textup{sink}\}$.
	\item $G_{e_5}$ is the sandpile group of $\X$. Note $A(e_5)=\{a_1,a_2,a_3\}$.
\end{itemize}
\end{eg}

%%%%%%%%%%%%%%%%%%%%%%%%%%%%%%%%%%%%%%%%%%%%%%%%%
\section{Classical and two-idempotent sandpile monoids}
\label{SubSecTwoIdemMonoids}

In this section we point out how two results for classical sandpile monoids generalize to directed sandpile monoids with two idempotents. If $\X$ is the graph for a classical sandpile monoid (i.e., $\X$ is undirected and $X$ is connected), the results from Section \ref{SubSecIdemStruct} yield the well-known result that $\M$ has one or two (usually two) idempotents. Let us quickly make this precise.

\begin{thm}
\label{ThmOneOrTwoIdemsForClassical}
Let $\X$ be undirected and $X$ connected. If $X$ consists of a single loopless vertex, then $\M$ has exactly one idempotent (namely, $0$). Otherwise, $\M$ has exactly two idempotents ($0$ and the identity of $\G$).
\end{thm}

\begin{proof}If $X$ consists of a single loopless vertex, then $X$ has no cycles so $\M$ has only one idempotent by part \ref{thmMGAcyclic} of Theorem \ref{BToumpMainThm1}. On the other hand, if $X$ consists of anything other than a single loopless vertex, then $X$ has a unique cyclic strong component, so by Theorem \ref{ThmBijection} $\M$ has exactly two idempotents, $0$ and the identity of $\G$.
\end{proof}

The case when $X$ consists of a single loopless vertex is relatively uninteresting---in this case, $\M=\G$ is just the cyclic group whose order is the outdegree of that vertex in $\X$. So all classical sandpile monoids of interest have exactly two idempotents. Sometimes the correct setting for generalizations of theorems about classical sandpile monoids to those based on directed graphs is the setting when $\M$ has exactly two idempotents. For example, the following theorem concerns an alternate definition of recurrence sometimes used for classical sandpile monoids.

\begin{thm}
\label{ThmRecurrenceEquiv}
Let $\M$ contain exactly two idempotents and let $u\in \M$. Then there exists $a\in\M$ with $a\neq 0$ such that $u\oplus a=u$ if and only if $u$ is recurrent.
\end{thm}

\begin{rmka}This theorem fails if we remove the hypothesis that $\M$ has exactly two idempotents. If $\M$ has one idempotent, $\M=\G$, so $0$ is the identity of $\G$ and $u\oplus a=u$ for some $a,u\in\M$ implies $a=0$. If $\M$ has more than 2 idempotents then there is an idempotent $a\in\M$ separate from $0$ and the identity of $\G$, $a$ being idempotent means $a\oplus a=a$, and $a$ is not recurrent.
\end{rmka}

To prove Theorem \ref{ThmRecurrenceEquiv} we will use the following theorem, whose proof we defer until Section \ref{SubSecMonoidsNotSPM}. %We will also see in Section \ref{SubSecMonoidsNotSPM} that this theorem also places a strong structural restriction on what monoids can be realized as sandpile monoids.

\begin{thma}(Theorem \ref{ThmUPlusAEqU}) Let $u,a\in\M$ such that $u\oplus a=u$. If $a\neq 0$, then $a$ cannot access 0.
\end{thma}

\begin{proof}[Proof of Theorem \ref{ThmRecurrenceEquiv}] If $\M$ contains exactly two idempotents then they are the identity of $\G$, which we denote by $e$, and $0$. 

If $u\oplus a=u$ for some $a\in \M$ with $a\neq 0$, then by Theorem \ref{ThmUPlusAEqU}, $a$ cannot access 0. We also have
\[
u=u\oplus a=u\oplus a\oplus a = \cdots = u\oplus e_a.
\]
Since $a$ cannot access $0$ and the only other idempotent in $\M$ is $e$, we have $e_a=e$. Thus $u\oplus e=u$, so $u$ is recurrent by part \ref{ThmMPlusEEqM} of Theorem \ref{thmThreePartRecurrence}.

On the other hand, if $u$ is recurrent then $u\oplus e=u$.
\end{proof}

Second, for classical sandpile monoids it is well known that the process of adding a non-zero element to itself and stabilizing repeatedly eventually produces a recurrent configuration.
%\cite{Intro}. 
For any directed graph $\X$, if $\M$ has only one idempotent, this is immediate (as all elements are recurrent), and if $\M$ has two idempotents, we have the following generalization.

\begin{thm}
\label{ThmEventuallyRecurrent}
Let $\M$ have exactly two idempotents and let $a\in \M$. If $a$ has a grain of sand on a vertex from which a cycle in $X$ is accessible, then adding $a$ to itself and stabilizing repeatedly will eventually reach a recurrent configuration. Conversely, if $a$ has no sand on any vertex from which a cycle in $X$ is accessible, then adding $a$ to itself and stabilizing repeatedly will never reach a recurrent configuration.
\end{thm}
\begin{proof}
If $a$ has a grain of sand on a vertex from which a cycle in $X$ is accessible, by part \ref{LemZeroAccCycle} of Theorem \ref{BToumpMainThm1} $a$ cannot access $0$. Since $\M$ has only two idempotents ($0$ and the identity of $\G$), $e_a$ must be the identity of $\G$.

For the converse, since $\M$ has two idempotents $X$ must contain at least one cycle. If $a$ has no sand on any vertex from which a cycle in $X$ is accessible, then the process of adding $a$ to itself and stabilizing repeatedly will never transfer any sand to a cycle. The converse thus follows from part  \ref{CorGInGOneGrainEveryCycle} of Theorem \ref{BToumpMainThm1}.
\end{proof}

In fact, this theorem is easy to generalize beyond two-idempotent sandpile monoids. Recall from Section \ref{SubSecIdemStruct} that the number of idempotents of $\M$ is equal to the number of filters on the set of cyclic strong components of $\X$.

\begin{thm}Let $a\in \M$. If $\mi(\supp(a))$ contains every cycle of $X$ then adding $a$ to itself and stabilizing repeatedly will eventually reach a recurrent configuration. Conversely, if there is a cycle of $X$ not contained in $\mi(\supp(a))$ then adding $a$ to itself and stabilizing repeatedly will never reach a recurrent configuration.
\end{thm}

\begin{proof}Let $e$ denote the identity of $\G$. By part  \ref{CorGInGOneGrainEveryCycle} of Theorem \ref{BToumpMainThm1}, $e$ has a grain of sand on every cycle of $X$, so $\supp(e)$ intersects every cyclic strong component of $\X$. By Theorem \ref{ThmBijection}, $e$ is the only idempotent of $\M$ whose support intersects every cyclic strong component of $\X$. If $\mi(\supp(a))$ contains every cycle of $X$, then $e_a$ also has this property, so $e_a=e$. On the other hand, if there is a cycle of $X$ not contained in $\mi(\supp(a))$ then adding $a$ to itself and stabilizing repeatedly will never transfer sand to this cycle. The converse follows from part  \ref{CorGInGOneGrainEveryCycle} of Theorem \ref{BToumpMainThm1}.
\end{proof}

%\begin{rmka}In many cases this process will enter $\G$ before it hits the identity of $\G$. Also, as soon as repetitions being appearing in the list $a,a\oplus a,a\oplus a\oplus a,\ldots$, we can be assured that we have reached a recurrent configuration.
%\end{rmka}

%%%%%%%%%%%%%%%%%%%%%%%%%%%%%%%%%%%%%%%%%%%%%%%%
\section{Monoids that are not sandpile monoids}
\label{SubSecMonoidsNotSPM}

In this section we give a new algebraic structural constraint for sandpile monoids and we exhibit two infinite families of finite commutative monoids which cannot be realized as sandpile monoids on any graph.

Every finite commutative group is isomorphic to the sandpile group of some graph. To see this, let $\X_k$ denote the graph which consists of two vertices (one ordinary vertex and the sink) and $k$ edges from the vertex to the sink. Then $\SPG{\X_k}=\SPM{\X_k}\cong\Z_k$. If we create the graph $\X$ by identifying the sinks of $\X_{k_1},\X_{k_2},\ldots,\X_{k_n}$, then 
\[
\SPM{\X}=\SPG{\X}\cong\Z_{k_1}\times \Z_{k_2} \times \cdots \times \Z_{k_n}.
\]
Thus, we also see that every finite commutative group is isomorphic to the sandpile {\em monoid} of some graph.

\begin{rmka}
Another way to prove that every finite commutative group is isomorphic to the sandpile group of some graph is by identifying undirected cycles at their sinks, as the sandpile group of the undirected cycle $C_n$ is $\Z_n$ \cite{DualGraphs}.
\end{rmka}

A natural question, then, is whether every finite commutative {\em monoid} is the sandpile monoid of some graph. The answer is no, and in this section we give two infinite families of monoids that cannot be realized as sandpile monoids on any graph. Here is our first family.

\begin{thm}
\label{ThmChainOfpElts}
Let $p$ be a prime number with $p>2$, and let $M_p$ denote a linearly ordered set of $p$ elements under the meet operation. Then $M_p$ is an idempotent commutative monoid which is not a sandpile monoid.\end{thm}

\begin{proof}
It is immediate that $M_p$ is an idempotent commutative monoid, and $|M_p|=p$ by definition.

Suppose $M_p$ is isomorphic to the sandpile monoid of some graph $\X$. Then $|\SPM{\X}|=p$ and thus, by \eqref{EqSizeOfM}, $X$ has exactly one vertex $v$ with $\outdeg(v)=p$ and the rest of the vertices in $X$ have out-degree 1. Furthermore, $\M$ is idempotent. 

Let $1_v$ denote the configuration on $\X$ consisting of one grain of sand on $v$ and none elsewhere. Then, since $\M$ is idempotent, $1_v\oplus 1_v = 1_v$, meaning that $1_v+1_v$ is an unstable configuration, which implies that $\outdeg(v)\leq 2$, a contradiction. 
\end{proof}

Next we prove Theorem \ref{ThmUPlusAEqU}, which provides an algebraic constraint that every sandpile monoid must satisfy.

\begin{thm}
\label{ThmUPlusAEqU}
Let $u,a\in\M$ such that $u\oplus a=u$. If $a\neq 0$, then $a$ cannot access 0.
\end{thm}

%\begin{rmka}There exist infinitely many finite commutative monoids $(M,+)$ where there exist $u,a,k\in M$ for which $u+a=u$, $a\neq 0$, and $a+k=0$. The monoid in the following table is one example, with $a=k=1$.

%\begin{center}
%\begin{tabular}{c|ccc}
%+&0&1&u\\
%\hline
%0&0&1&u\\
%1&1&0&u\\
%u&u&u&u\\
%\end{tabular}
%\end{center}
%Also, note that the hypothesis that $u\oplus a=u$ and $a\neq 0$ need not imply that $u$ is recurrent, as it does for classical sandpile monoids (Theorem \ref{ThmRecurrenceEquiv}), for if $a$ is any non-zero idempotent besides the identity of $\G$, then a cannot be recurrent and $a\oplus a=a$.
%\end{rmka}

\begin{proof}
Suppose that $u,a\in\M$, $u\oplus a=u,$ and $a\neq 0$ (i.e., $\supp(a)$ is nonempty). Suppose further that $a$ can access $0$. Then for all $w\in \supp(a)$, by part \ref{LemZeroAccCycle} of Theorem \ref{BToumpMainThm1}, there is no path from $w$ to any cycle in $X$. 

Let $w$ denote any member of $\supp(a)$ for which there is no path of length 1 or greater in $X$ from any member of $\supp(a)$ to $w$. To see that such a $w$ exists, suppose not. Then we have $w_1\in \supp(a)$, to which there is a path from $w_2 \in \supp(a)$, to which there is a path from $w_3\in \supp(a)$, etc. Since $\supp(a)$ is finite, we eventually get a cycle in $X$ that contains a member of $\supp(a)$, a contradiction.

Let $a_w$, $u_w$, and $(u\oplus a)_w$ indicate the number of grains of sand the configurations $a$, $u$, and $u\oplus a$, respectively, have on the vertex $w$. We have $u_w=(u\oplus a)_w$ and since $w\in \supp(a)$ and both $a$ and $u$ are stable we also have 
\begin{align}
\label{EqPfw}
1\leq a_w < \outdeg(w), \\
%\label{EqPfw2}
\notag
0\leq u_w < \outdeg(w).
\end{align}
Now, note that the only vertices that may have grains of sand transferred to them by topplings in the process of stabilizing $u+a$ are vertices to which there is a path in $X$ from some member of $\supp(a)$. Thus, during the stabilization of $u+a$, $w$ will not have any grains of sand transferred to it through toppling. Therefore, no grains may be transferred away from $w$ in the stabilization of $u+a$ either, for if some were then we would have $(u\oplus a)_w < u_w$ by (\ref{EqPfw}) (in particular, $a_w,u_w<\outdeg(w)$), which would imply $u\oplus a\neq u$.

Now, since no grains of sand are transferred to or from $w$ in the stabilization of $u+a$, we have $(u\oplus a)_w = u_w+a_w$. Since $(u\oplus a)_w=u_w$, we have $a_w=0$, a contradiction.
\end{proof}

\begin{rmk}
\label{RemarkMonoidNotSPMTest}
Theorem \ref{ThmUPlusAEqU}, said another way, says that if $(M,+)$ is a finite commutative monoid with identity $0$ and elements $u,a,k$ satisfying $u+a=u$, $a\neq 0$, and $a+k=0$, then $M$ cannot be realized as a sandpile monoid on any graph. This theorem gives a straightforward way to show that a monoid cannot be a sandpile monoid, although there are monoids---for example, those in Theorem \ref{ThmChainOfpElts}---which cannot be sandpile monoids that are not detected by this test.
\end{rmk}

%\begin{cor}If $a\in \M$ is idempotent and $a\neq 0$, then $a$ cannot access $0$.
%\end{cor}
%
%Note that this corollary also follows from part \ref{LemZeroAccCycle} of Theorem \ref{BToumpMainThm1} and Remark \ref{RmkNoInitialAbnormalForIdem}, as every nonzero idempotent has sand on some cycle of $X$.

Now, if $\M$ is a classical sandpile monoid, then $\M$ contains exactly one or two idempotents, so another question is whether every finite commutative monoid containing exactly one or two idempotents is isomorphic to the sandpile monoid of some graph. First, if $M$ is a finite commutative monoid with exactly one idempotent, then $M$ is actually a group, so the answer in this case is yes. However, as the next theorem shows, there are infinitely many finite commutative monoids with exactly two idempotents that are not the sandpile monoids of any (directed or undirected) graph.

\begin{thm}
\label{ThmTwoIdemsNotSPM}
Let $n>2$ and let $G=\{0,1,\ldots,n-2\}$ be the cyclic group of order $n-1$ (with operation denoted by $+$). Let $(M,+)$ be $G\cup\{\infty\}$ under the same operation as $G$, with
\[
\infty+\infty=\infty, \up{ and } g+\infty=\infty+g=\infty \textup{ for all }g\in G.
\]
Then $M$ is a commutative monoid of order $n$ with exactly two idempotents, and $M$ is not a sandpile monoid.
\end{thm}

\begin{proof}The idempotents of $M$ are $0$ (the identity) and $\infty$. Let $u=\infty, a=1,k=n-2$. Then $u+a=u$, $a\neq 0$, and $a+k=0$, so by Remark \ref{RemarkMonoidNotSPMTest}, $M$ is not the sandpile monoid of any graph.
\end{proof}

\begin{rmka}In \cite{BToump} it is shown that the lattice of idempotents of a sandpile monoid is distributive. Therefore, any finite commutative monoid whose lattice of idempotents is not distributive is not a sandpile monoid. In particular, any finite non-distributive lattice is a finite (idempotent) commutative monoid that is not a sandpile monoid. The two families of monoids in this section are both families of monoids whose lattices of idempotents {\em are} distributive, and yet the monoids themselves are still not sandpile monoids. 
%, so we see that the condition that the lattice of idempotents of a monoid being distributive is not sufficient to guarantee that a monoid is a sandpile monoid.
\end{rmka}

%%%%%%%%%%%%%%%%%%%%%%%%%%%%%%%%%%%%%%%%%%%%%%%%%%%%%%%%%%
\section{Sink-distance-regular directed graphs}
\label{SecStronglyRegular}

%In this section, all graphs are undirected, connected, regular of degree $k$, and have no loops or multiple edges. 

%
%Sandpile computations are sometimes easy in regular graphs.  For example:
%
%\begin{thm}
%Let $X$ be an (undirected) regular graph of degree $k$. Form the graph $\X$ from $X$ by joining every vertex in $V$ to the sink with one edge.
%Let $m \in \M$ be the element in which every vertex holds exactly one grain of sand.
%Then $\MAX =k m$ is the identity element of $\G$.
%\end{thm}
%
%\begin{proof}
%Given the unstable configuration $\MAX + m = (k+1)m$, we may topple every vertex in $V$ once.
% Each vertex sends out $k+1$ grains of sand and receives $k$ grains from those adjacent to it.
% One grain from each vertex enters the sink, so that after toppling each vertex in $V$, every vertex in $V$ holds
% $k$ grains of sand. Thus $\MAX \oplus m = \MAX$. 
% Therefore, for any positive integer
% $j, \MAX \oplus j m = \MAX$.
% In particular,
% $\MAX \oplus \MAX = \MAX \oplus k m = \MAX$
% and so $\MAX$ is idempotent.
% Since $\MAX\in\G$ and groups have only one idempotent, $\MAX$ is the identity of $\G$.
%\end{proof}
%
% If the graph $\X$ is regular, then the group identity may be a bit more complicated, but if $\X$ is {\em distance-regular} we can use a certain submonoid of $\M$ to compute the identity element.

In \cite[Section 5]{BiggsDRG} Biggs gives an explicit combinatorial description of the sandpile group identity for any (undirected) distance-regular graph. In this section we study a family of directed graphs that generalizes the family of distance-regular graphs and we give an explicit combinatorial description (Theorem \ref{ThmDistRegularId}) of the sandpile group identity for every member of this family. For distance-regular graphs, our description recovers that of Biggs. Our family also includes many other families of interest, including iterated wheels (Example \ref{ExIteratedWheels}), regular trees (Example \ref{ExRegularTrees}), and regular tournaments (Example \ref{ExRegularTournament}).

First we recall what it means for a graph to be distance-regular.

\begin{defn}Let $Y$ be a finite undirected graph. For vertices $x, y$ of $Y$, the {\em distance} $\delta(x,y)$ from $x$ to $y$ is the length of the shortest path from $x$ to $y$. If $Y$ is connected, the {\em diameter} of $Y$, denoted $\diam(Y)$, is given by $\diam(Y)=\max \{ \delta(x,y) : x, y \in Y\}.$
\end{defn}

Let $Y$ be a finite, connected, undirected graph with no loops and with at most one edge between each pair of vertices. Given two vertices $x, y$ of $Y$ of distance $\delta(x,y) = i$, define $c_i(x,y), a_i(x,y)$, and $b_i(x,y)$ to be the number of vertices $z$ adjacent to $y$ such that $\delta(x,z)$ is $i-1, i$, or $i+1,$ respectively.
(The integer $c_i(x,y)$ counts the vertices $z$ that are adjacent to $y$ but ``closer" to $x$, and $b_i(x,y)$ counts those $z$  adjacent to $y$ that are ``beyond" $y$.) Note that, by definition, for a fixed pair of vertices $x,y$ with $\delta(x,y)=i$, the sum $c_i(x,y) + a_i(x,y) + b_i(x,y)$ is the degree of $y$, and in general the degree of a vertex $x$ is $b_0(x,x)$.

\begin{defn}Let $Y$ be a finite, connected, undirected graph with no loops and with at most one edge between each pair of vertices. $Y$ is said to be {\em distance-regular} if the integers $c_i(x,y), a_i(x,y),$ and $b_i(x,y)$ are constants, independent of the choice of vertices $x,y$ of distance $i$. That is, a graph is distance-regular if, for each integer $i$, for all vertices $x_1,y_1,x_2,y_2$ such that $\delta(x_1,y_1)=\delta(x_2,y_2)=i$, we have $a_i(x_1,y_1)=a_i(x_2,y_2)$, $b_i(x_1,y_1)=b_i(x_2,y_2)$, and $c_i(x_1,y_1)=c_i(x_2,y_2)$.
\end{defn}

Thus, if a graph is distance-regular we can simply refer to the numbers $a_i, b_i$, and $c_i$. A distance-regular graph is certainly regular, as the degree of a vertex $x$ is $b_0(x,x)$, and $b_0(x,x)=b_0(y,y)$ for all vertices $x,y$. Basic material on distance-regular graphs may be found in Chapters 20 and 21 of \cite{Biggs}, the text \cite{BCN}, Sections 1.6 and 4.5 of \cite{Godsil}, and Chapter 21 of \cite{MacWilliams}. It is often fruitful in the study of a distance-regular graph $Y$ to begin by fixing a vertex $s$ and partitioning the vertices of $Y$ based on their distance from $s$, defining $\Gamma_i = \{y \in Y : \delta(y,s)=i\},$ for $0 \le i \le \diam(Y).$ In such a partition, vertices of $\Gamma_i$ can only be adjacent to vertices in $\Gamma_{i-1},\Gamma_i$, or $\Gamma_{i+1}$.

We now consider the following generalization to directed graphs. We remind the reader that the digraph $\X$ may have loops and multiple edges. 

\begin{defn}For vertices $x,y\in\X$, the {\em distance} $\delta(x,y)$ from $x$ to $y$ is the length of the shortest directed path (of length 0 or greater) from $x$ to $y$ in $\X$, if such a path exists, and is $\infty$ otherwise. If $i$ is a nonnegative integer, the {\em $i$th neighborhood} of $y$, denoted $\Gamma_i(y)$, is
\[
\Gamma_i(y)=\{x\in \X: \delta(x,y)=i\}.
\]
\end{defn}
Denote the sink of $\X$ by $s$, and denote $\Gamma_i(s)$ by $\Gamma_i$. Since there is a path from every vertex to the sink and $\X$ is finite, $d=\max_{v\in\X}\{\delta(v,s)\}$ is an integer, and $\{\Gamma_0,\Gamma_1,\ldots,\Gamma_d\}$ forms a partition of the vertices of $\X$, with $\Gamma_0=\{s\}$. If $(v,w)$ is a directed edge of $\X$ from $v$ to $w$, we call $v$ and $w$ the {\em tail} and the {\em head}, respectively, of the edge $(v,w)$. 

\begin{defn}
\label{DefLocallyDistanceRegular} We call $\X$ {\em sink-distance-regular} if the sink of $\X$ has outdegree 0 and the neighborhoods $\{\Gamma_0,\Gamma_1,\ldots,\Gamma_d\}$ of the sink satisfy the following properties:
\begin{enumerate}
	\item if $v\in\Gamma_i$ and $(v,w)$ is a directed edge of $\X$, then $w\in\Gamma_{i-1}$, $w\in\Gamma_i$, or $w\in\Gamma_{i+1}$ (where $\Gamma_{d+1}$ is interpreted as $\{\}$), and
	\item there exist collections of nonnegative integers $\{a_i\}_{i=1}^d$ and positive integers $\{b_i\}_{i=1}^{d-1}$ and $\{c_i\}_{i=1}^d$, such that
	\begin{enumerate}
	%\item the sink is the head of precisely $b_0$ edges whose tails are in $\Gamma_1$,
	\item if $v\in \Gamma_1$, then $v$ is the tail of precisely $c_1$ edges whose heads are the sink,
	\item if $v\in \Gamma_i$ where $i>1$, then $v$ is the tail of precisely $c_i$ edges whose heads are in $\Gamma_{i-1}$ and $v$ is the head of precisely $c_i$ edges whose tails are in $\Gamma_{i-1}$, and
	\item if $v\in \Gamma_i$ where $i\geq 1$, then $v$ is the tail of precisely $a_i$ edges whose heads are in $\Gamma_i$, $v$ is the head of precisely $a_i$ edges whose tails are in $\Gamma_i$, $v$ is the tail  of precisely $b_i$ edges whose heads are in $\Gamma_{i+1}$, and $v$ is the head of precisely $b_i$ edges whose tails are in $\Gamma_{i+1}$ (where $b_d$ is interpreted as $0$).
\end{enumerate}
\end{enumerate}
\end{defn}

\begin{rmk}
When $\X$ is sink-distance-regular, the neighborhoods of the sink form a directed version of the concept of ``equitable partition" appearing in pages 195--198 of \cite{Godsil}. If $Y$ is a distance-regular graph, % of diameter $d$ with parameters $a_i$, $b_i$, and $c_i$, then after choosing a vertex $s$ of $Y$ to be the sink, partitioning the vertices of $Y$ based on their distance from $s$ (with $\Gamma_i = \{y \in Y | \delta(s,y)=i\},$ for $0 \le i \le d$), 
then after choosing a vertex of $Y$ to be the sink and removing the outward edges from the sink we obtain a graph that is sink-distance-regular. However, the converse is not true even for undirected graphs---every distance-regular graph is regular, whereas there exist sink-distance-regular graphs (where every edge not including the sink is undirected) which are not regular---see, for example, the graphs in Examples \ref{ExIteratedWheels} and \ref{ExRegularTrees}.
\end{rmk}

Sink-distance-regular graphs are far more common than distance-regular graphs, and it is simple to construct examples of them one neighborhood at a time, beginning with $\Gamma_0$. For a small example of a sink-distance-regular graph that consists mostly of directed edges, see Figure \ref{FigDirectedDRG} (where each undirected edge is to be interpreted as a pair of directed edges, one in each direction). This graph is sink-distance-regular, with $d=2$, $\Gamma_1=\{v_1,v_2\}$, $\Gamma_2=\{v_3,v_4,v_5,v_6\}$, $c_1=2$, $a_1=1$, $b_1=2$, $c_2=1$, and $a_2=1$.

\begin{figure}[htb]
\begin{center}
\begin{tikzpicture}[>=stealth']
\tikzset{VertexStyle} = [shape = circle]
\SetUpVertex[Lpos=-90]
\SetUpEdge[lw= 1.2pt]
\SetGraphUnit{1.3} 
%\GraphInit[vstyle=Classic] 
\Vertex[Math,L={\up{sink}},Lpos=-180]{v0}
\NOEA[Math,L=v_1](v0){v1}
\SOEA[Math,L=v_2,Lpos=90](v0){v2}
\NOEA[Math,L=v_3,Lpos=90](v1){v3}
\EA[Math,L=v_4,Lpos=-135](v1){v4}
\EA[Math,L=v_5,Lpos=135](v2){v5}
\SOEA[Math,L=v_6](v2){v6}
\tikzset{EdgeStyle/.style={->}}
\Edge(v1)(v0)
\Edge[style={bend right}](v1)(v0)
\Edge(v2)(v0)
\Edge[style={bend left}](v2)(v0)
\Edge(v1)(v5)
\Edge(v2)(v4)
\Edge(v3)(v4)
\Edge(v4)(v5)
\Edge(v5)(v6)
\Edge[style={bend right}](v6)(v3)
\Loop[dist=1cm,dir=NOWE,style={very thick,->}](v1)
\Loop[dist=1cm,dir=SOWE,style={very thick,->}](v2)
\Edge(v4)(v1)
\Edge(v5)(v2)
\tikzset{EdgeStyle/.style={-}}
\Edge(v1)(v3)
\Edge(v2)(v6)
\end{tikzpicture}
\caption{A small sink-distance-regular graph $\X$.}
\label{FigDirectedDRG}
\end{center}
\end{figure}
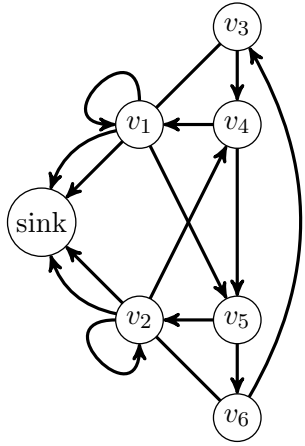

Now, let $\X$ be a sink-distance-regular graph with $d\geq 1$. 
%
%\begin{rmka}Every distance-regular graph has an ``adjacency diagram"  (or ``distribution diagram" \cite[p. 62]{BCN}), which is a directed graph on $d+1$ vertices created by collapsing each set $\Gamma_i$  to a single vertex (where $\Gamma_0=\{\textup{sink}\}$).
%The adjacency diagram and the associated adjacency algebra motivate the following exploration of the sandpile monoid of distance-regular graphs.
%For details on this algebra, see any of the references on distance-regular graphs cited above.
%However, the material below is self-contained---while it is motivated by the adjacency diagram, it does not rely on any specific results from the theory of distance-regular graphs.
%\end{rmka}
For $i\geq 1$, let $\gamma_i$ denote the configuration on $\X$, stable or not, with one grain of sand on each vertex of $\Gamma_i$ and none elsewhere. Note that if $v\in\Gamma_i$ and $i\geq 1$, then $\outdeg(v)=a_i+b_i+c_i$---thus every vertex of $\Gamma_i$ has the same outdegree. We write $\outdeg(\Gamma_i)$ for this common outdegree. Also, if $m$ is a sandpile configuration on $\X$ and $j$ is a nonnegative integer, we denote the state (stable or not)
\[
\underbrace{m+m+\cdots+m}_j
\]
by $jm$. 

If $x$ is an unstable configuration on $\X$ in which each vertex of $\Gamma_i$ holds at least $\outdeg(\Gamma_i)$ grains of sand, let $T_i(x)$ denote the configuration (stable or not) created by toppling each vertex of $\Gamma_i$ exactly once. Then it follows from the definitions of the parameters $a_i$, $b_i$, and $c_i$ that if $d=1$, then
\[
T_1(x)=T_d(x)=x-c_1\gamma_1,
\]
and if $d>1$, then
\begin{align*}
T_1(x)&=x-(c_1+b_1)\gamma_1 + c_2\gamma_2,\\
T_d(x)&=x-c_d\gamma_d+b_{d-1}\gamma_{d-1}, \up{ and}\\
T_i(x)&=x-(c_i+b_i)\gamma_i + b_{i-1}\gamma_{i-1}+c_{i+1}\gamma_{i+1}
\end{align*}
for $1<i<d$.

%Since $\X$ is distance-regular and $c_i+a_i+b_i=k$, for $1 < i < d$ we have
%$$T_i({\bf x}) = {\bf x}+b_{i-1}{\gamma_{i-1}} -(c_i+b_i){\gamma_{i}} +c_{i+1}{\gamma_{i+1}}.$$
%For $i=1$, since $c_1=1$, we have
%$$T_1({\bf x}) ={\bf x}-  (1+b_1){\gamma_{1}} +c_{2}{\gamma_{2}}.  $$
%For $i=d$ we have
%$$T_d({\bf x}) ={\bf x}+ b_{d-1}{\gamma_{d-1}} - c_d{\gamma_{d}} .$$
%In the case $d=1$ the above is to be interpreted as
%$$$
%From the discussion of the toppling operators $T_i$, it follows that
%the set
%$$\{ \sum_{i=1}^d  \alpha_i \gamma_i |\alpha_i \in \Z, 0 \le \alpha_i < k\}$$
% is a submonoid of $\M$. It turns out that this submonoid contains the identity of $\G$, which we describe in Theorem \ref{ThmDistRegularId}. 
Let $U_i$ denote the following composition of these topplings:
$$U_i = T_dT_{d-1} \cdots T_{i+1}T_i,$$
where $U_d=T_d$. We will use the following lemma regarding the $U_i$ to prove that our description of the sandpile group identity of $\X$ in Theorem \ref{ThmDistRegularId} is correct.
%%%%%%%%
\begin{lem}
\label{LemUiFormulas}
Let $1\leq i\leq d$. Suppose $x$ is an unstable sandpile configuration in which
each vertex of $\Gamma_i$ is ready to topple (that is, each vertex of $\Gamma_i$ holds at least $\outdeg(\Gamma_i)$ grains of sand) and, for all $j$ with $i< j \leq d$, each vertex of $\Gamma_{j}$ holds at least $\outdeg(\Gamma_j)-c_{j}$ grains of sand. Then $U_i$ is a legal sequence of topplings to apply to $x$ and, for $1 < i \leq d$ we have
$$U_{i}(x) = x +b_{i-1}\gamma_{i-1} - c_{i}\gamma_{i},$$
and for $i=1$ we have
$$U_1(x) = x - c_1\gamma_1.$$
\end{lem}

\begin{proof}Certainly $T_i$ may be applied to $x$. If $i<d$, this moves $c_{i+1}$ grains of sand onto each vertex of $\Gamma_{i+1}$, so we may apply $T_{i+1}$ to $T_i(x)$. In general, if $i<j<d$, none of $T_{i}, T_{i+1},\ldots T_{j-1}$ affect the number of grains of sand on the vertices of $\Gamma_{j+1}$, and since $T_j$ moves $c_{j+1}$ grains of sand onto each vertex of $\Gamma_{j+1}$, $T_{j+1}$ may be applied to $T_jT_{j-1}\cdots T_i(x)$. This justifies the legalities of the topplings used in $U_i(x)$.

Next we verify the formulas for the $U_i(x)$. The result is immediate for $d=1$ and a simple calculation verifies $d=2$, so for the remainder of the proof assume $d>2.$ 

First, $T_d=U_d$ and the formula for $T_d(x)$ agrees with the one we claimed for $U_d(x)$. Next, for $i=d-1$ we have
\[
U_{d-1}(x)=T_{d}(T_{d-1}(x)) = x+ b_{d-2}{\gamma_{d-2}} -c_{d-1}{\gamma_{d-1}},
\]
as claimed. Next, if $1<i<d-1$, then a straightforward induction shows that
\[
T_jT_{j-1}\cdots T_{i+1}T_i(x) = x + b_{i-1}\gamma_{i-1} - c_i\gamma_i - b_j\gamma_j + c_{j+1}\gamma_{j+1}
\]
for all $i<j<d$. In particular,
\[
T_{d-1}T_{d-2}\cdots T_{i+1}T_i(x) = x + b_{i-1}\gamma_{i-1} - c_i\gamma_i - b_{d-1}\gamma_{d-1} + c_{d}\gamma_d.
\]
Applying $T_d$ to this we obtain
\begin{align*}
U_i(x) &= 
x + b_{i-1}\gamma_{i-1} - c_i\gamma_i - b_{d-1}\gamma_{d-1} + c_{d}\gamma_d + b_{d-1}{\gamma_{d-1}} - c_d{\gamma_{d}} \\
&= x + b_{i-1}\gamma_{i-1} - c_i\gamma_i,
\end{align*}
as claimed. Finally, for $i=1$ we have
\begin{align*}
U_1(x) = U_2(T_1(x)) &= U_2(x-(c_1+b_1){\gamma_{1}} +c_{2}{\gamma_{2}})\\
&= x-  (c_1+b_1){\gamma_{1}} +c_{2}{\gamma_{2}} + b_1\gamma_1 - c_2 \gamma_2\\
&= x - c_1\gamma_1. \qedhere
\end{align*}
%%%%%%%%%%%%
 
%If $1<i<d-1$:
%$$T_{i+1}T_{i}(x) = x+ b_{i-1}{\gamma_{i-1}} -c_{i}{\gamma_{i}} -b_{i+1}{\gamma_{i+1}} +c_{i+2}{\gamma_{i+2}}.$$

%If $i=1$ then
%$$T_{2}T_1(x) = x-{\gamma_{1}} -b_2{\gamma_{2}} +c_{3}{\gamma_{3}}.$$
\end{proof}

%%%%%%%%%%%%
Recall that $c_i>0$ for all $1\leq i \leq d.$ Now, we use the floor function to define the parameter
\[
n_d = \left\lfloor\frac{\outdeg(\Gamma_d)-1}{c_d} \right\rfloor.
\]
Note that $\outdeg(\Gamma_d)\geq 1,$ so $n_d\geq 0$. Thus $0\le n_dc_d \le \outdeg(\Gamma_d)-1$ but $n_dc_d +c_d > \outdeg(\Gamma_d)-1.$ In other words, the configuration $(n_dc_d) \gamma_d$ is stable but adding another $c_d$ grains of sand to each vertex of $\Gamma_d$ makes it unstable.

Define, furthermore, for $1 \le i<d$,
\[
n_i = \left\lfloor\frac{\outdeg(\Gamma_i)-1+n_{i+1}b_i}{c_i}\right\rfloor.
\]
Thus $n_i\geq 0$, $n_ic_i-n_{i+1}b_i \leq \outdeg(\Gamma_i) -1$, and $n_ic_i-n_{i+1}b_i+c_i > \outdeg(\Gamma_i) -1.$ Also let $n_{d+1}=0$. 

\begin{thm}
\label{ThmDistRegularId}
Let $\X$ be a sink-distance-regular directed graph, with $d,\{a_i\}_{i=1}^d,\{b_i\}_{i=1}^d,\{c_i\}_{i=1}^d$ as in Definition \ref{DefLocallyDistanceRegular}, and $\{n_i\}_{i=1}^{d+1}$ as above. 
Then the identity element of $\G$ is the element $e\in \M$ given by
$$e= \sum_{i=1}^{d} (n_ic_i -n_{i+1}b_i) \gamma_i.$$
\end{thm}

\begin{proof}

First we show that $e$ really is an element of $\M$. For all $1\leq i\leq d$, we have 
\[
n_ic_i-n_{i+1}b_i+c_i > \outdeg(\Gamma_i) -1,
\]
so
\[
n_ic_i-n_{i+1}b_i \geq \outdeg(\Gamma_i) - c_i,
\]
and since $c_i\leq \outdeg(\Gamma_i)$,
\[
n_ic_i-n_{i+1}b_i \geq 0,
\]
so $e$ has a nonnegative number of grains of sand on each vertex. Also for all $1\leq i\leq d$,
$$
n_ic_i-n_{i+1}b_i \leq \outdeg(\Gamma_i) -1,
$$
so $e$ is stable. Thus $e\in \M$.

Now, to show that $e$ is the identity of $\G$, by Remark \ref{RmkNoInitialAbnormalForIdem} it suffices to show that $e$ is idempotent and that $e$ has a grain of sand on every cycle of $\X$. 
First, we show $e$ is idempotent. 

By the definition of the parameters $n_i$, $e$ is only barely stable, in the sense that if, for any $1\leq i\leq d$, $c_i$ additional grains of sand are added to each vertex of $\Gamma_i$, the resulting configuration is unstable and we may apply $U_i.$ Lemma \ref{LemUiFormulas} then gives 
\begin{equation*}
U_i(e + m + c_i \gamma_i) = e+m+b_{i-1}\gamma_{i-1}
%\label{eqUofe}
\end{equation*}
for any $m\in\M$ and any $i$ with $1 < i \le d.$ 
For $i=1$ we have
\[
U_1(e + m+ c_1\gamma_1) = e+m
\]
for any $m\in \M$. If $k$ is a nonnegative integer, we denote the composition of the toppling operator $U_i$ with itself $k$ times by $U_i^k$. We compute $e \oplus e$ by toppling the configuration $e+e$ as follows.

If $d=1$, then $e=n_1c_1\gamma_1$, so $e+e=2n_1c_1\gamma_1$, so $U_1$ may be applied $n_1$ times to $e+e$. Thus $U_1^{n_1}(e+e)=U_1^{n_1}(e+n_1c_1\gamma_1)=e$, so $e$ is idempotent.

If $d=2$, a similar argument shows that $U_2$ may be applied to $e+e$ $n_2$ times, and $U_1$ may then be applied $n_1$ times to the resulting state, and that $U_1^{n_1}(U_2^{n_2}(e+e)) = e$, so $e$ is idempotent.

Let $d>2$. Then a straightforward induction shows that, if $1<j\leq d$, then $U_j^{n_j}\cdots U_d^{n_d}$ is a legal sequence of topplings to apply to $e+e$, and 
$$
U_j^{n_j}\cdots U_d^{n_d} (e+e) = e + n_{j-1}c_{j-1}\gamma_{j-1} + \sum_{i=1}^{j-2}(n_ic_i-n_{i+1}b_i)\gamma_i.
$$
In particular, for $j=2$, we have
\[
U_2^{n_2}\cdots U_d^{n_d} (e+e) = e + n_{1}c_{1}\gamma_{1},
\]
so $U_1^{n_1}U_2^{n_2}\cdots U_d^{n_d} (e+e) = U_1^{n_1} (e + n_{1}c_{1}\gamma_{1})=e$. Thus $e$ is idempotent.

%older less formal argument
%The unstable configuration $e+e$ has $2n_dc_d$ grains of sand on each vertex of $\Gamma_d.$
%We first apply the toppling operator $(U_d)^{n_d}$ (that is, $U_d$ $n_d$ times) to remove $n_dc_d$ grains of sand from $\Gamma_d$ so that $\Gamma_d$ has $n_dc_d$ grains of sand on each vertex, and so those vertices are stable.
%In doing this, we have added $n_db_{d-1}$ grains of sand to each vertex in $\Gamma_{d-1}.$
%Now apply the operator $(U_{d-1})^{n_{d-1}}$ to this configuration.
%The application of $(U_{d-1})^{n_{d-1}}$ removes $n_{d-1}c_{d-1}$ grains of sand from each vertex of $\Gamma_{d-1}$ and adds $n_{d-1}b_{d-2}$ grains to each vertex of $\Gamma_{d-2}$. Now the vertices of both $\Gamma_d$ and $\Gamma_{d-1}$ are stable, the number of grains of sand on each vertex of
%$\Gamma_d$ is still $n_dc_d$, the number of grains of sand on each vertex of
% $\Gamma_{d-1}$ is $n_{d-1}c_{d-1}-n_db_{d-1},$ and the vertices of $\Gamma_{d-2}$ are ready to topple.
%
%Successive applications of $(U_{i})^{n_i}$ remove $n_ic_i$ grains of sand from each member of $\Gamma_i$ and pass on $n_ib_{i-1}$ grains of sand to each member of $\Gamma_{i-1}.$  They leave vertices of $\Gamma_j$ unchanged if $j > i$ or $j < i-1.$
%
%Thus, applying (\ref{eqUofe}) repeatedly, we obtain $U_1^{n_1}U_2^{n_d} \cdots U_i^{n_i} \cdots U_d^{n_d}(e+e)=e,$ i.e., $e$ is idempotent.

Finally, we show that $e$ has a grain of sand on every cycle of $X$. First, if $1\leq i<d$, then
$n_ic_i-n_{i+1}b_i+c_i \geq \outdeg(\Gamma_i),$ so
\begin{align*}
n_ic_i-n_{i+1}b_i &\geq \outdeg(\Gamma_i) -c_i \\
&=(a_i+b_i+c_i)-c_i\\
&=a_i+b_i\\
&\geq 1
\end{align*}
because $b_i\geq 1$ and $a_i\geq 0$. Thus every vertex of $\Gamma_i$ holds at least one grain of sand in $e$, for $1\leq i<d$. 

Suppose now for the sake of contradiction that there exists a cycle $C$ of $X$ such that no vertex of $C$ holds a grain of sand in $e$. Then every vertex of $C$ must lie in $\Gamma_d$, so $\Gamma_d$ contains a cycle, and so $a_d\geq 1$. But we also have $n_dc_d + c_d \geq \outdeg(\Gamma_d)$, so
\begin{align*}
n_dc_d &\geq \outdeg(\Gamma_d) - c_d \\
&=(a_d+c_d) - c_d \\
&=a_d\\
&\geq 1.
\end{align*}
Thus every vertex of $\Gamma_d$ holds at least one grain of sand on in $e$, contradicting our assumption.

Thus $e$ has a grain of sand on every cycle of $X$.
\end{proof}

\begin{eg}Let $Y$ be a distance-regular graph of diameter $d$. Then $Y$ is regular, say of degree $k$. Form the graph $\X$ by choosing any vertex of $Y$ to be the sink and deleting all outward edges from the sink. Then $\X$ is sink-distance-regular, where the constants $\{a_i\}_{i=1}^d,\{b_i\}_{i=1}^d,\{c_i\}_{i=1}^d$ agree with those that arise from the distance-regularity of $Y$. The identity $e$ of $\G$ is then 
$$e= \sum_{i=1}^{d} (n_ic_i -n_{i+1}b_i) \gamma_i,$$
where $n_{d+1}=0$,
\[
n_d = \left\lfloor\frac{k-1}{c_d} \right\rfloor,
\]
and, for $1 \le i<d$,
\[
n_i = \left\lfloor\frac{k-1+n_{i+1}b_i}{c_i}\right\rfloor.
\]
This agrees with Biggs's description of the identity of $\G$ \cite{BiggsDRG}.
\end{eg}

We now give some examples to show how Theorem \ref{ThmDistRegularId} generalizes beyond distance-regular graphs.

\begin{eg}\label{ExIteratedWheels}

Let $n$ and $d$ be integers with $n\geq 2$ and $d\geq 1$. Let $Y$ be the undirected prism graph $Y_{n,d}=C_n\times P_d$ (where $C_n$ is the undirected $n$-cycle and $P_d$ is the path graph on $d$ vertices), with the $C_n$'s of $Y$ arranged in the plane in a concentric fashion. 
%Arrange the vertices of $Y$ in a planar fashion, in a series of concentric cycles. 
Form the graph $\X$ by adding a central vertex (the sink) and adding an edge from each of the innermost vertices of $Y$ to the sink. We call $\X$ the $d$-iterated $n$-wheel; when $d=1$ $\X$ is the $n$-wheel, whose sandpile group was considered as an example in \cite{DualGraphs}. Figure \ref{FigForExIteratedWheels} depicts the $3$-iterated $5$-wheel (where each undirected edge is to be interpreted as a pair of directed edges, one in each direction).

\begin{figure}[htb]
\begin{center}
\begin{tikzpicture}[>=stealth']
\SetUpVertex[Lpos=-90]
\SetUpEdge[lw= 1.2pt]
\SetGraphUnit{2} 
\tikzset{VertexStyle} = [shape = circle]
%\GraphInit[vstyle] 
\Vertex[Math,L={\up{sink}}]{v0}

\Vertex[Math,L={v_{1}},x=1.42658477444273,y=0.463525491562421]{v1}
\Vertex[Math,L={v_{2}},x=0.000000000000000,y=1.50000000000000]{v2}
\Vertex[Math,L={v_{3}},x=-1.42658477444273,y=0.463525491562421]{v3}
\Vertex[Math,L={v_{4}},x=-0.881677878438710,y=-1.21352549156242]{v4}
\Vertex[Math,L={v_{5}},x=0.881677878438709,y=-1.21352549156242]{v5}
\Vertex[Math,L={v_{6}},x=2.56785259399692,y=0.834345884812358]{v6}
\Vertex[Math,L={v_{7}},x=0.000000000000000,y=2.70000000000000]{v7}
\Vertex[Math,L={v_{8}},x=-2.56785259399691,y=0.834345884812358]{v8}
\Vertex[Math,L={v_{9}},x=-1.58702018118968,y=-2.18434588481236]{v9}
\Vertex[Math,L={v_{10}},x=1.58702018118968,y=-2.18434588481236]{v10}
\Vertex[Math,L={v_{11}},x=3.70912041355110,y=1.20516627806230]{v11}
\Vertex[Math,L={v_{12}},x=0.000000000000000,y=3.90000000000000]{v12}
\Vertex[Math,L={v_{13}},x=-3.70912041355110,y=1.20516627806230]{v13}
\Vertex[Math,L={v_{14}},x=-2.29236248394065,y=-3.15516627806230]{v14}
\Vertex[Math,L={v_{15}},x=2.29236248394065,y=-3.15516627806230]{v15}

\tikzset{EdgeStyle/.style={->}}
\Edge(v1)(v0)
\Edge(v2)(v0)
\Edge(v3)(v0)
\Edge(v4)(v0)
\Edge(v5)(v0)
\tikzset{EdgeStyle/.style={-}}
\Edge(v1)(v2)
\Edge(v1)(v5)
\Edge(v1)(v6)
\Edge(v2)(v3)
\Edge(v2)(v7)
\Edge(v3)(v4)
\Edge(v3)(v8)
\Edge(v4)(v5)
\Edge(v4)(v9)
\Edge(v5)(v10)
\Edge(v6)(v7)
\Edge(v6)(v10)
\Edge(v6)(v11)
\Edge(v7)(v8)
\Edge(v7)(v12)
\Edge(v8)(v9)
\Edge(v8)(v13)
\Edge(v9)(v10)
\Edge(v9)(v14)
\Edge(v10)(v15)
\Edge(v11)(v12)
\Edge(v11)(v15)
\Edge(v12)(v13)
\Edge(v13)(v14)
\Edge(v14)(v15)

\end{tikzpicture}
\caption{The $3$-iterated $5$-wheel $\X$; $\Gamma_1=\{v_1,\ldots,v_5\}$, $\Gamma_2=\{v_6,\ldots,v_{10}\}$, $\Gamma_3=\{v_{11},\ldots,v_{15}\}$.}
\label{FigForExIteratedWheels}
\end{center}
\end{figure}
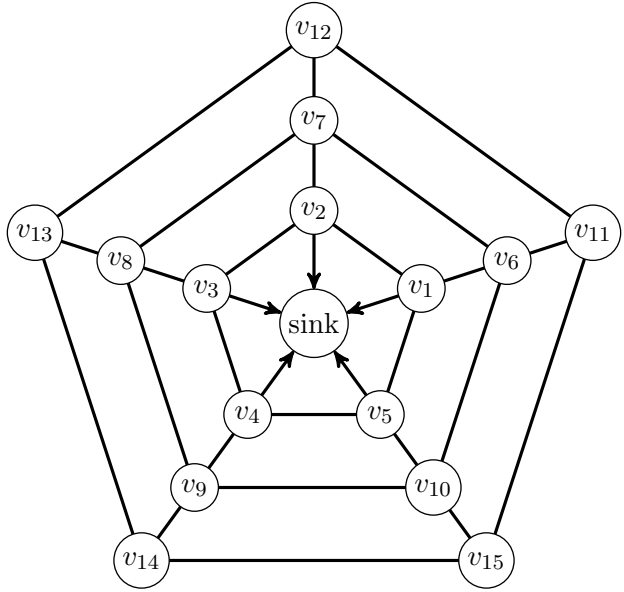

$\X$ is sink-distance-regular. For $1\leq i\leq d$, $\Gamma_i$ consists of the vertices of the $i$th cycle of $\X$ (counted starting from the sink and moving outward), $c_i=1$, and $a_i=2$. We also have $b_i=1$ for all $1\leq i<d$. A straightforward induction shows that $n_i=2+3(d-i)$ for $1\leq i\leq d$, so $n_dc_d=2$ and $n_ic_i -n_{i+1}b_i = 3$ for $1\leq i<d$. The identity of $\G$ is therefore
\begin{align*}
e&= \sum_{i=1}^{d} (n_ic_i -n_{i+1}b_i) \gamma_i \\
% &= 2\gamma_d+\sum_{i=1}^{d-1} (2+3(d-i) - (2+3(d-(i+1))))\gamma_i \\
 &= 2\gamma_d+\sum_{i=1}^{d-1} 3\gamma_i \\
 &= \MAX.
\end{align*}
\end{eg}

\begin{eg}\label{ExRegularTrees} Let $n$ and $d$ be integers with $n\geq 3$ and $d\geq 2$. Let $Y$ be the $n$-regular (undirected) tree of depth $d$---that is, let $Y$ be a finite rooted tree where each non-leaf vertex has degree $n$ (so the root has $n$ children and each non-root, non-leaf vertex has $n-1$ children) and where the path from any leaf to the root has $d-1$ edges. Form the graph $\X$ by adding a new vertex (the sink) and adding $n-1$ directed edges from each leaf of $Y$ to the sink. The sandpile group of $\X$ was studied in \cite{ToumpakariTrees}. A variant was studied in \cite{LevineTree}. 
%The graph $\X$ obtained from the $2$-regular tree of depth $3$ is depicted in Figure \ref{FigForRegularTrees} (where each undirected edge is to be interpreted as a pair of directed edges, one in each direction).
%
%\begin{figure}[htb]
%\begin{center}
%\begin{tikzpicture}[>=stealth']
%\SetUpVertex[Lpos=-90]
%\SetUpEdge[lw= 1.2pt]
%\SetGraphUnit{1.2} 
%\tikzset{VertexStyle} = [shape = circle]
%%\GraphInit[vstyle] 
%\Vertex[Math,L=v_1]{v1}
%\NOEA[Math,L=v_2](v1){v2}
%\SOEA[Math,L=v_3](v1){v3}
%\NOEA[Math,L=v_4](v2){v4}
%\EA[Math,L=v_5](v2){v5}
%\EA[Math,L=v_6](v3){v6}
%\SOEA[Math,L=v_7](v3){v7}
%\SOEA[Math,L=\up{sink}](v5){v0}
%\tikzset{EdgeStyle/.style={-}}
%\Edge(v1)(v2)
%\Edge(v1)(v3)
%\Edge(v2)(v4)
%\Edge(v2)(v5)
%\Edge(v3)(v6)
%\Edge(v3)(v7)
%\tikzset{EdgeStyle/.style={->}}
%\Edge(v4)(v0)
%\Edge[style={bend left}](v4)(v0)
%\Edge(v5)(v0)
%\Edge[style={bend right}](v5)(v0)
%\Edge(v6)(v0)
%\Edge[style={bend left}](v6)(v0)
%\Edge(v7)(v0)
%\Edge[style={bend right}](v7)(v0)
%\end{tikzpicture}
%\caption{The graph $\X$ obtained from the $2$-regular tree of depth $3$; $\Gamma_1=\{v_4,v_5,v_6,v_7\}$, $\Gamma_2=\{v_2,v_3\}$, $\Gamma_3=\{v_1\}$.}
%\label{FigForRegularTrees}
%\end{center}
%\end{figure}

$\X$ is sink-distance-regular. For $1\leq i\leq d$, $\Gamma_i$ consists of the vertices of distance $d-i$ from the root and $a_i=0$. We also have $c_i=n-1$ and $b_i=1$ for $1\leq i<d$, and $c_d=n$. A straightforward induction shows that $n_d=0$ and $n_i=1$ for $1\leq i<d$, so identity of $\G$ is 
\begin{align*}
e&= \sum_{i=1}^{d} (n_ic_i -n_{i+1}b_i) \gamma_i \\
 &= 0\gamma_d+ (n-1)\gamma_1 + \sum_{i=1}^{d-2} (n-2)\gamma_i,
\end{align*}
the configuration of $\M$ that holds no grains on the root, $n-1$ grains on each child of the root, and $n-2$ grains on every other non-sink vertex of $\X$.
\end{eg}

\begin{eg}\label{ExRegularTournament} Let $k$ be a positive integer and let $Y$ be a regular tournament on $2k+1$ vertices---that is, let $Y$ be a directed graph with $2k+1$ vertices where there is exactly one directed edge between every pair of vertices and where each vertex has in-degree $k$ and out-degree $k$. Let $r$ be a positive integer and form the graph $\X$ by adding a vertex (the sink) and adding $r$ directed edges from each vertex of $Y$ to the sink. The graph $\X$ obtained from a regular tournament on 5 vertices (with $r=1$) is depicted in Figure \ref{FigForRegularTournaments}.

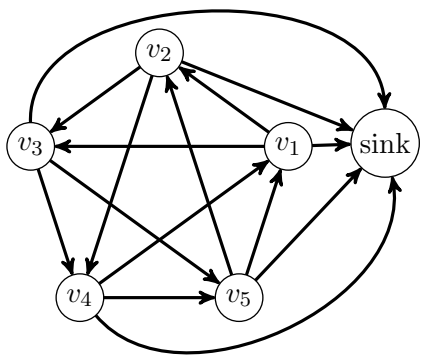
\begin{figure}[htb]
\begin{center}
\begin{tikzpicture}[>=stealth']
\SetUpVertex[Lpos=-90]
\SetUpEdge[lw= 1.2pt]
\SetGraphUnit{1.3} 
\tikzset{VertexStyle} = [shape = circle]
%\GraphInit[vstyle] 
\Vertex[Math,L={v_{1}},x=1.71190172933128,y=0.556230589874906]{v1}
\Vertex[Math,L={v_{2}},x=0.000000000000000,y=1.80000000000000]{v2}
\Vertex[Math,L={v_{3}},x=-1.71190172933128,y=0.556230589874906]{v3}
\Vertex[Math,L={v_{4}},x=-1.05801345412645,y=-1.45623058987491]{v4}
\Vertex[Math,L={v_{5}},x=1.05801345412645,y=-1.45623058987491]{v5}
\Vertex[Math,L={\up{sink}},x=3,y=.6]{v0}
\tikzset{EdgeStyle/.style={->}}
\Edge(v1)(v2)
\Edge(v2)(v3)
\Edge(v3)(v4)
\Edge(v4)(v5)
\Edge(v5)(v1)
\Edge(v1)(v3)
\Edge(v2)(v4)
\Edge(v3)(v5)
\Edge(v4)(v1)
\Edge(v5)(v2)
\Edge(v1)(v0)
\Edge(v2)(v0)
\Edge[style={bend left=90}](v3)(v0)
\Edge[style={bend right=75}](v4)(v0)
\Edge(v5)(v0)
\end{tikzpicture}
\caption{The graph $\X$ obtained from a regular tournament on 5 vertices ($r=1$).}
\label{FigForRegularTournaments}
\end{center}
\end{figure}

$\X$ is sink-distance regular, with $d=1$. $\Gamma_1$ consists of all non-sink vertices of $\X$, $a_1=k$, and $c_1=r$, so the identity of $\M$ is
\[
e=n_1c_1\gamma_1=\left(\left\lfloor \frac{k+r-1}{r} \right\rfloor r \right)\gamma_1.
\]
When $r=1$ this is $\MAX$.
\end{eg}

\begin{eg}\label{ExSinkDistRegularDirected}
Let $\X$ be the graph in Figure \ref{FigDirectedDRG}, where each undirected edge is to be interpreted as a pair of directed edges, one in each direction. $\X$ is sink-distance-regular, with $d=2$, $\Gamma_1=\{v_1,v_2\}$, $\Gamma_2=\{v_3,v_4,v_5,v_6\}$, $c_1=2$, $a_1=1$, $b_1=2$, $c_2=1$, and $a_2=1$. We have
\[
n_3=0, \quad n_2=1, \quad n_1=3,
\]
so the identity of $\G$ is
\[
e = (3\cdot2-1\cdot2)\gamma_1 + (1\cdot1)\gamma_2 = 4 \gamma_1 + 1 \gamma_2 = \MAX.
\]
\end{eg}

%As an example, consider the (8,4) Johnson graph, whose vertices consist of the 70 subsets of $\{1,2,...,8\}$ of size four. Two vertices in this graph are adjacent if the subsets intersect in exactly three elements.
%This is a distance-regular graph of diameter 4 and parameters
%$$k=b_0 = 16, \ b_1=9,\  b_2=4,\  b_3=1$$
%$$c_1 = 1,  \ c_2 = 4,\  c_3 = 9,\  c_4 = 16.$$
%Johnson graphs are described in Section 1.6 of \cite{Godsil} and Section 21.6 of \cite{MacWilliams}.
%
%Fix one of the vertices to be the sink. Then $\Gamma_1$ and $\Gamma_3$ each have 16 elements, $\Gamma_2$ has 36 elements, and $\Gamma_4$ has just one (the complement of the sink). We compute
%$$n_4 = \left\lfloor \frac{15}{16} \right\rfloor = 0,$$
%$$n_3 = \left\lfloor \frac{15}{9} \right\rfloor = 1,$$
%$$n_2 = \left\lfloor \frac{15+4}{4} \right\rfloor = 4,$$
%and
%$$n_1 = \left\lfloor \frac{15+36}{1} \right\rfloor = 51.$$
%The coefficients for the identity of the sandpile group are
%$$n_1c_1 - n_2b_1 =  51-36=15;
%n_2c_2 - n_3b_2 = 16-4=12;
%n_3c_3 - n_4b_3 = 9-0=9 ;
%n_4c_4 = 0,$$
%so the identity is
%$$e = 15 \gamma_1 + 12 \gamma_2 + 9 \gamma_3 + 0 \gamma_4.$$
%In this case, the identity is not $\MAX$; indeed,
%$\MAX-e =   0 \gamma_1+3 \gamma_2 + 6 \gamma_3 + 15 \gamma_4,$
%so by Corollary \ref{CorMRecurrentIffAccFromSPId} there are in fact $(4^{36})(7^{16})(16)
%\approx  2.51 \times 10^{36}$ elements of the sandpile group accessible from $e$ without toppling.

{\medskip\noindent\bf Acknowledgments.} We thank L\'aszl\'o Babai and the anonymous referees for their comments and suggestions, which have helped us substantially improve our results in Section \ref{SecStronglyRegular} and our presentation of the material throughout this paper.

%\nocite{DharSurvey}
%Bibtex:
\bibliographystyle{plain}
\bibliography{SandpileBib}

\end{document}